\documentclass{amsart}

\usepackage{geometry}
\usepackage{amssymb}
\usepackage{verbatim}
\usepackage{graphicx}
\usepackage{color}
\usepackage{esint}

\newtheorem{prop}{Proposition}[section]
\newtheorem{teo}[prop]{Theorem}

\newtheorem{lemma}[prop]{Lemma}
\newtheorem{cor}[prop]{Corollary}
\theoremstyle{definition}
\newtheorem*{definizione}{Definition}
\newtheorem{oss}[prop]{Remark}
\newtheorem*{ack}{Acknowledgments}

\usepackage[colorlinks=true,urlcolor=blue, citecolor=red,linkcolor=blue,
linktocpage,pdfpagelabels, bookmarksnumbered,bookmarksopen]{hyperref}
\usepackage[hyperpageref]{backref}

\date{\today}

\numberwithin{equation}{section}

\title[Lower bound with topological constraints]{An optimal lower bound\\ in fractional spectral geometry\\ for planar sets with topological constraints}

\author[Bianchi]{Francesca Bianchi}
\address[F.\ Bianchi]{Dipartimento di Scienze Matematiche, Fisiche e Informatiche
	\newline\indent
	Universit\`a di Parma
	\newline\indent
	Parco Area delle Scienze 53/a, Campus, 43124 Parma, Italy}
\email{francesca.bianchi@unipr.it}
\author[Brasco]{Lorenzo Brasco}
\address[L.\ Brasco]{Dipartimento di Matematica e Informatica
	\newline\indent
	Universit\`a degli Studi di Ferrara
	\newline\indent
	Via Machiavelli 35, 44121 Ferrara, Italy}
\email{lorenzo.brasco@unife.it}

\begin{document}

\begin{abstract}
We prove a lower bound on the first eigenvalue of the fractional Dirichlet-Laplacian of order $s$ on planar open sets, in terms of their inradius and topology. The result is optimal, in many respects. In particular, we recover a classical result proved independently by Croke, Osserman and Taylor, in the limit as $s$ goes to $1$. The limit as $s$ goes to $1/2$ is carefully analyzed, as well. 
\end{abstract}

\maketitle

\begin{center}
\begin{minipage}{10cm}
\small
\tableofcontents
\end{minipage}
\end{center}

\section{Introduction}

	\subsection{Goal of the paper}
	In this paper, we pursue our investigation on geometric estimates for the following sharp fractional Poincar\'e constant
\begin{equation}\label{lambda1s}
	\lambda^s_1(\Omega):=\inf_{u\in C^\infty_0(\Omega)\setminus\{0\}}\dfrac{[u]^2_{W^{s,2}(\mathbb{R}^2)}}{\|u\|^2_{L^2(\Omega)}},
\end{equation}
on planar open sets $\Omega\subseteq\mathbb{R}^2$. Here the parameter $0<s<1$ represents a fractional order of differentation and the quantity $[\,\cdot\,]_{W^{s,2}(\mathbb{R}^2)}$ is given by
	\[
[u]_{W^{s,2}(\mathbb{R}^2)}=\left(\iint_{\mathbb{R}^2\times\mathbb{R}^2}\frac{|u(x)-u(y)|^2}{|x-y|^{2+2\,s}}\,dx\,dy\right)^\frac{1}{2},\qquad \mbox{ for every } u\in C^\infty_0(\mathbb{R}^2).
\]
All functions in $C^\infty_0(\Omega)$ are considered as elements of $C^\infty_0(\mathbb{R}^2)$, by extending them to be zero outside $\Omega$. The infimum in \eqref{lambda1s} can be equivalently performed on the space $\widetilde W^{s,2}_0(\Omega)$. The latter is defined as the closure of $C^\infty_0(\Omega)$ in the fractional Sobolev-Slobodecki\u{\i} space
\[
W^{s,2}(\mathbb{R}^2)=\Big\{u\in L^2(\mathbb{R}^2)\, :\, [u]_{W^{s,2}(\mathbb{R}^2)}<+\infty\Big\},
\] 
endowed with its natural norm. 
Whenever the infimum \eqref{lambda1s} becomes a minimum on this larger space $\widetilde{W}^{s,2}_0(\Omega)$, the quantity $\lambda_1^s(\Omega)$ will be called {\it first eigenvalue of the fractional Dirichlet-Laplacian of order $s$ on} $\Omega$.
\par
The constant $\lambda^s_1(\Omega)$ can be seen as a fractional counterpart of 
\[
\lambda_1(\Omega):=\inf_{u\in C^\infty_0(\Omega)\setminus\{0\}}\dfrac{\|\nabla u\|^2_{L^2(\Omega)}}{\|u\|^2_{L^2(\Omega)}},
\]
which coincides with the bottom of the spectrum of the more familiar Dirichlet-Laplacian on $\Omega$. The link between $\lambda_1^s$ and $\lambda_1$ can be made more precise by recalling that 
\[
\lim_{s\nearrow 1}(1-s)\,[u]^2_{W^{s,2}(\mathbb{R}^2)}=C\,\|\nabla u\|^2_{L^2(\Omega)},\qquad \mbox{ for every }u\in C^\infty_0(\Omega),
\]
for some universal constant $C>0$, see \cite{BBM} or \cite[Chapter 3]{EE}.
\vskip.2cm\noindent
The present paper is a continuation of our previous work \cite{BB}, to which we refer for more background material. In particular, we still focus on getting lower bounds on $\lambda_1^s(\Omega)$, in terms of the {\it inradius} of $\Omega$, which is defined by
\[
r_\Omega:=\sup\Big\{r>0\, :\, \exists\, x_0\in \Omega \mbox{ such that } B_r(x_0)\subseteq \Omega\Big\},
\] 
where $B_r(x_0)$ is the open disk of center $x_0$ and radius $r$.\par
In \cite[Theorem 1.1]{BB}, extending a classical result of Makai \cite{Mak} and Hayman \cite{Ha} valid for $\lambda_1$ (see also \cite{An, BCproc} and \cite{BC}), we showed that we have 
\[
\lambda_1^s(\Omega)\ge \mathcal{C}_s\,\left(\frac{1}{r_\Omega}\right)^{2\,s},
\]
for every {\it simply connected} open set $\Omega\subseteq\mathbb{R}^2$ with finite inradius and for every $1/2<s<1$. Here the constant $\mathcal{C}_s$ depends on $s$ only and it has the following asymptotic behaviours\footnote{Here, the writing ``$f\sim g$ for $x\to x_0$'' as to be intended in the following sense
	\[
	0<\liminf_{x\to x_0} \frac{f(x)}{g(x)}\le \limsup_{x\to x_0} \frac{f(x)}{g(x)}<+\infty.
	\]}
	\[
	\mathcal{C}_s\sim (2\,s-1)\ \mbox{ for } s\searrow \frac{1}{2}\qquad \mbox{ and }\qquad 
\mathcal{C}_s\sim \frac{1}{1-s}\ \mbox{ for } s\nearrow 1.
\]
Moreover, we showed by means of a counterexample, that for $0<s\le 1/2$ such a lower bound is not possible (see \cite[Theorem 1.3]{BB}).
\vskip.2cm\noindent
In the present paper we considerably extend this result, by considering open connected planar sets having {\it non-trivial topology}. 
More precisely, we will work with the following class of sets:	
\begin{definizione}\label{def: k connesso con compattificazione}
Let us indicate by $(\mathbb{R}^2)^*$ the {\it one-point compactification} of $\mathbb{R}^2$, i.e. the compact space obtained by adding to $\mathbb{R}^2$ the point at infinity. We say that an open connected set $\Omega\subseteq\mathbb{R}^2$ is {\it multiply connected of order $k$} if its complement in $(\mathbb{R}^2)^*$  has $k$ connected components. When $k=1$, we will simply say that $\Omega$ is {\it simply connected}.
\end{definizione}
We thus seek for an estimate of the type
\[
\lambda_1^s(\Omega)\ge \mathcal{C}_{s,k}\,\left(\frac{1}{r_\Omega}\right)^{2\,s},
\]
for open multiply connected sets of order $k$ in the plane. In light of the simply connected case recalled above, we can directly restrict our analysis to the case $1/2<s<1$ only.
\subsection{The Croke-Osserman-Taylor inequality} 
For the classical case of $\lambda_1$, the first lower bound of this type is due to Osserman. Notably, \cite[Theorem p. 546]{Os} shows that
\[
\lambda_1(\Omega)\ge \min\left\{\frac{1}{4},\, \frac{1}{k^2}\right\}\,\left(\frac{1}{r_\Omega}\right)^2,
\]
for every $\Omega\subseteq\mathbb{R}^2$ open multiply connected of order $k$.
The proof by Osserman is based on a refinement of the so-called {\it Cheeger's inequality}, in conjunction with Bonnesen--type inequalities.
\par
It turns out that the estimate by Osserman does not display the sharp dependence on the topology of the sets, i.e. the term $1/k^2$ is sub-optimal, as $k$ diverges to $\infty$.
Indeed, the result by Osserman has been improved by Taylor in \cite[Theorem 2]{Ta}, showing that 
\[
\lambda_1(\Omega)\ge\frac{C}{k}\,\left(\frac{1}{r_\Omega}\right)^2,
\]
for some constant $C>0$ which is not made explicit in \cite{Ta}. The dependence on $k$ is now optimal, for $k$ going to $\infty$.
The proof by Taylor is quite sophisticated and completely different from Osserman's one: it is based on estimating the first eigenvalue with mixed boundary conditions (i.e. Dirichlet and Neumann) of a set, in terms of the {\it capacity} of the ``Dirichlet region''. 
Such an estimate is achieved by means of heat kernel estimates. This method is connected with Taylor's work \cite{Ta2} on the {\it scattering length} of a positive potential, which acts as a perturbation of the Laplacian (see also \cite{Siu} for a generalization to the case of the fractional Laplacian).  
We will come back in a moment on Taylor's proof, since our main result will be based on the same arguments.
\par 
An improvement of Taylor's estimate has been given by Croke, who gives the explicit lower bound 
\[
\lambda_1(\Omega)\ge\frac{1}{2\,k}\,\left(\frac{1}{r_\Omega}\right)^2,
\] 
for $k\ge2$ (see \cite[Theorem]{Cr}). The proof by Croke is more elementary and based on refining Osserman's argument.
\par
Finally, for completeness we mention \cite[Theorem 3]{GR} by Graversen and Rao, which proves the following lower bound
\[
\lambda_1(\Omega)\ge\frac{C_k}{r_\Omega^2},\qquad\mbox{where }\qquad C_k=\left\{\begin{array}{cc}
1/4, & \mbox{ if } k=1,\\
&\\
\dfrac{A}{k\,\log k}, & \mbox{ if }k\ge 2,
\end{array}
\right.
\] 
for some $A>0$ (see \cite[Theorem 3]{GR}). Their result is slightly worse when compared with the ones by Croke and Taylor. We notice that the proof in \cite{GR} uses techniques from the theory of Brownian motion, which are quite close to the ideas by Taylor. 

\subsection{Main results}	
Our goal is to generalize the Croke--Osserman--Taylor result to the setting of fractional Sobolev spaces. We also want to discuss the optimality of the estimate we obtain, with respect to the parameters $k$ and $s$.

\begin{teo}[Main Theorem]
\label{teo:main}
	Let $1/2<s<1$, there exists a constant $\vartheta_s>0$ such that for every $\Omega\subseteq\mathbb{R}^2$ open multiply connected set of order $k\in\mathbb{N}\setminus\{0\}$, we have 
	\begin{equation}
	\label{fracOTC}
	\lambda_1^s(\Omega) \ge \frac{\vartheta_s}{k^s}\,\left(\frac{1}{r_\Omega}\right)^{2\,s}.
	\end{equation}
	Moreover, the constant $\vartheta_s$ has the following asymptotic behaviours
	\[
	\vartheta_s\sim (2\,s-1)\ \mbox{ for } s\searrow \frac{1}{2}\qquad \mbox{ and }\qquad 
\vartheta_s\sim \frac{1}{1-s}\ \mbox{ for } s\nearrow 1.
\]
\end{teo}	
The next result shows that the estimate \eqref{fracOTC} is sharp, {\it apart from the evaluation of the absolute constant\footnote{This is a quotation from Taylor's paper, see \cite[page 452]{Ta}.}}.
\begin{teo}[Optimality] 
\label{teo:optimal}
The following facts hold:
\begin{enumerate}
\item for every $\Omega\subseteq\mathbb{R}^2$ open set, we have
\[
\limsup_{s\nearrow 1} (1-s)\,\lambda_1^s(\Omega)\le \frac{1}{2}\,\lambda_1(\Omega).
\]
Thus, the estimate \eqref{fracOTC} is sharp in its dependence on $s\nearrow 1$.
In particular, by taking the limit as $s$ goes to $1$ in \eqref{fracOTC}, we get the classical Croke-Osserman-Taylor inequality, possibly with a worse constant;
\vskip.2cm
\item let $1/2<s<1$, there exists a sequence $\{\Omega_k\}_{k\in\mathbb{N}\setminus\{0\}}\subseteq \mathbb{R}^2$ of open sets such that $\Omega_k$ is multiply connected of order $k$
\[
r_{\Omega_k}\le C\qquad \mbox{ and }\qquad  \limsup_{k\to\infty} k^s\,\lambda_1^s(\Omega_k)<+\infty.
\]
Thus the estimate \eqref{fracOTC} is sharp in its dependence on $k\to \infty$;
\vskip.2cm
\item for every $k\in\mathbb{N}\setminus\{0\}$, there exists $\Theta_k\subseteq \mathbb{R}^2$ an open multiply connected set of order $k$, such that 
\[
r_{\Theta_k}<+\infty\qquad \mbox{ and }\qquad \limsup_{s\searrow \frac{1}{2}} \frac{\lambda_1^s(\Theta_k)}{2\,s-1}<+\infty.
\]
Thus, the estimate \eqref{fracOTC} is sharp in its dependence on $s\searrow 1/2$.
\end{enumerate}
\end{teo}

\subsection{Comments on the proofs}
As anticipated above, the statement of Theorem \ref{teo:main} contains our previous result \cite[Theorem 1.1]{BB} as a particular case. Indeed, the latter  was concerned with simply connected sets, i.e. with the case $k=1$. However, the proof given here is completely different: the elegant and elementary argument used in \cite{BB}, taken from \cite{Ha}, crucially exploited the simple connectedness and would not work here. Actually, a much more sophisticated argument is needed now. We also point out that it seems extremely complicated to adapt the proof by Osserman (and Croke), because a genuine Cheeger's inequality is still missing in the fractional case. 
\par
The general strategy for proving Theorem \ref{teo:main} will be the same as in \cite{Ta}. However, even if we closely follow Taylor's ideas, some important modifications are needed and new technical difficulties arise. In addition, we tried to simplify and/or expand some of the arguments contained in \cite{Ta}. We now expose the overall strategy of the proof and highlight the main changes needed to cope with the fractional case:
\begin{enumerate}
\item at first, we tile the whole plane $\mathbb{R}^2$ by a family of squares $\{\mathcal{Q}_{i,j}\}_{(i,j)\in \mathbb{Z}^2}$. By observing that for every $u\in C^\infty_0(\Omega)$ we have
\[
[u]^2_{W^{s,2}(\mathbb{R}^2)}\ge \sum_{(i,j)\in\mathbb{Z}^2} \iint_{\mathcal{Q}_{i,j}\times \mathcal{Q}_{i,j}}\frac{|u(x)-u(y)|^2}{|x-y|^{2+2\,s}}\,dx\,dy,
\]
we can reduce the problem to proving a ``regional'' fractional Poincar\'e inequality on squares such that $\mathcal{Q}_{i,j}\cap \Omega\not=\emptyset$. Of course, the main difficulty lies in getting such an inequality with an explicit constant, which only depends on the geometry (i.e. on $r_\Omega$) and topology (i.e. on $k$) of the open set $\Omega$;
\vskip.2cm
\item this type of Poincar\'e inequality is possible only if $u\in C^\infty_0(\Omega)$ vanishes on ``sufficiently large portions'' of $\mathcal{Q}_{i,j}$, for every square $\mathcal{Q}_{i,j}$ intersecting $\Omega$. Here ``largeness'' has to be intended in the sense of {\it fractional Sobolev capacity}. Thus, the first important step of this strategy is to prove a {\it Maz'ya-type Poincar\'e inequality} on a square, for functions vanishing on a compact subset $\Sigma$ of positive fractional capacity (see Proposition \ref{lm:poincare_cap}). The constant in such an inequality can be estimated from below in terms of the capacity of the ``Dirichlet region'' $\Sigma$;
\vskip.2cm
\item the second step consists in converting the previous {\it analytic} estimate into a {\it geometric} one. In other words, we have to bound from below the fractional capacity of the ``Dirichlet region'' $\Sigma$ in terms of some of its geometric features. This can be done by using orthogonal projections, which enable a dimensional reduction argument. In the two-dimensional setting, this permits to estimate the fractional capacity of $\Sigma$ in terms of the length of its orthogonal projection on a line. Such an estimate is possible as soon as {\it points have positive fractional capacity in dimension $1$}. This happens precisely if and only if $s>1/2$;
\vskip.2cm
\item the previous two points clarify that, in order to conclude the proof, we need to know that in each square $\mathcal{Q}_{i,j}$ intersecting $\Omega$, there is a ``Dirichlet region'' $\Sigma_{i,j}$ having at least an orthogonal projection ``large enough'', i.e. with a length depending on $r_\Omega$ and $k$ in a uniform way.
\par
Here we crucially rely on a topological argument by Taylor, that we have called ``Taylor's fatness lemma'' (see Lemma \ref{lm:rettangolo distante}). In a nutshell, it asserts that any multiply connected planar set $\Omega$ with finite inradius has a ``locally uniformly fat'' complement. This means that, if we choose the size of $\mathcal{Q}_{i,j}$ sufficiently large (in terms of $r_\Omega$ and $k$), then this square must contain a portion of $\mathbb{R}^2\setminus\Omega$ which has an orthogonal projection with
\[
\mbox{ length}\simeq \sqrt{k}\,r_\Omega,
\]
in a universal fashion, i.e. no matter the location of the square. 
\end{enumerate}
Differently from Taylor's paper, we work here with a {\it variational} definition of (fractional) capacity (see for example \cite{AFN, Ri, SX, Warma}), which appears more natural and well-adapted to the problem. This permits to prove the Poincar\'e inequality at point (2) above in an elementary way, by avoiding both the heat kernel estimate and the reference 
to an eigenvalue with mixed boundary conditions used in \cite{Ta}. Both points would have been problematic (or at least complicated) in the fractional setting.
Also, we point out that our proof of the Maz'ya--type Poincar\'e inequality is genuinely {\it nonlinear} in nature.
\par 
As for point (3): with respect to \cite{Ta}, we expand the explanations and try to make the geometric estimates as much quantitative as possible. There is in addition a technical difficulty linked to the fractional case: in the classical case treated by Taylor, one essential ingredient of the dimensional reduction argument is the following simple algebraic fact
\[
|\partial_\omega u|^2\le |\nabla u|^2,\qquad \mbox{ for every }\omega\in\mathbb{S}^{N-1}.
\]
In the fractional case, there is no direct analogue of this simple formula.
Nevertheless,
it is possible to give a sort of fractional counterpart of this property (see Proposition \ref{prop:direzioni}), but the proof is by far less straighforward: in order to prove it, we find it useful to resort to some {\it real interpolation techniques} (see also \cite[Appendix B]{BCV}). These permit to ``localize the nonlocality'', in a sense. We think this part to be interesting in itself.
\par
The ``fatness lemma'' of point (4) would be just a topological fact and could be directly recycled in the fractional case.
However, in \cite{Ta} this is not explicitly  stated in the form that can be found below. Here as well, we tried to add some details and precisions. We believe that the final outcome should be useful to have a better understanding of Taylor's proof.
\par
Finally, in all the estimates presented above, a great effort is needed in order to obtain the correct asymptotic behaviour of $s\mapsto \vartheta_s$ claimed in Theorem \ref{teo:main}. In particular, getting the sharp asymptotic behaviour for $s\searrow 1/2$ requires a very careful analysis. Accordingly, proving that the set $\Theta_k$ in Theorem \ref{teo:optimal} provides the sharp decay rate at $0$ needs quite refined (though elementary) estimates. We point out that this part is new already for the simply connected  case, previously considered in \cite{BB}.

\subsection{Plan of the paper}
All the needed notations are settled in Section \ref{sec:2}. Here we also state and prove Taylor's fatness lemma. Section \ref{sec:3} contains some technical facts on fractional Sobolev spaces which are useful for our main result, though hard to trace back in the literature. The uninterested reader may skip this part on a first reading. We then introduce the relevant notion of fractional capacity in Section \ref{sec:4} and prove the main building blocks for obtaining the fractional Croke-Osserman-Taylor inequality. Sections \ref{sec:5} and \ref{sec:6} contain the proofs of Theorems \ref{teo:main} and \ref{teo:optimal}, respectively. Finally, the paper is concluded by two appendices.

\begin{ack}
We thank Eleonora Cinti, Stefano Francaviglia and Francesca Prinari for some useful discussions. We also thank Rodrigo Ba\~nuelos for pointing out the reference \cite{Siu}.
The results of this paper have been announced during the mini-workshop ``{\it A Geometric Fairytale full of Spectral Gaps and Random Fruit\,}'', held at the Mathematisches Forschungsinstitut Oberwolfach in December 2022, as well as during the meeting ``{\it PDEs in Cogne: a friendly meeting in the snow\,}'', held in Cogne in January 2023.
We wish to thank the organizers 
for the kind invitations and the nice working atmosphere provided during the stayings.
\par
 F.B. is a member of the Gruppo Nazionale per l'Analisi Matematica, la Probabilit\`a
e le loro Applicazioni (GNAMPA) of the Istituto Nazionale di Alta Matematica (INdAM). Both authors have been financially supported by the {\it Fondo di Ateneo per la Ricerca} {\sc FAR 2020} of the University of Ferrara. 
\end{ack}

\section{Preliminaries}
\label{sec:2}

\subsection{Notation}
For every $\alpha\in\mathbb{R}$, we denote its {\it integer part} by 
\[
\big\lfloor\alpha\big\rfloor=\max\Big\{n\in\mathbb{Z}\, :\, \alpha\ge n\Big\}.
\]
We recall that
\begin{equation}
\label{mantissa}
\alpha-1\le \big\lfloor\alpha\big\rfloor\le \alpha,\qquad \mbox{ for every } \alpha\in\mathbb{R}.
\end{equation}
For $r>0$ and $x_0\in\mathbb{R}^N$, we will indicate 
\[
B_r(x_0)=\Big\{x\in\mathbb{R}^N\, :\, |x-x_0|<r\Big\},
\]
and
\[
Q_r(x_0)=\prod_{i=1}^N (x_0^i-r,x_0^i+r),\qquad \mbox{ where } x_0=(x_0^1,\dots,x_0^N).
\]
When the center $x_0$ coincides with the origin, we will simply write $B_r$ and $Q_r$, respectively. We will indicate by $\omega_N$ the $N-$dimensional Lebesgue measure of $B_1$.
\par
For completeness, we also recall the following classical definition from point-set topology.
\begin{definizione}
Let $K\subset \mathbb{R}^N$, we say that $K$ is a {\it continuum} if it is a non-empty compact and connected set.
\end{definizione}
For every $\omega\in\mathbb{S}^{N-1}$, we will indicate by
\[
\langle \omega\rangle^\bot=\Big\{x\in\mathbb{R}^N\, :\, \langle x,\omega\rangle=0\Big\},
\]
the orthogonal space to $\omega$. We will also set 
\begin{equation}
\label{projection}
\begin{array}{ccccc}
\Pi_\omega&:&\mathbb{R}^N &\to& \langle \omega\rangle^\bot\\
&& x&\mapsto & x-\langle x,\omega\rangle\,\omega,
\end{array}
\end{equation}
i.e. this is the orthogonal projection on $\langle \omega\rangle^\bot$. In particular, for $N=2$, if we indicate by $\mathbf{e}_1=(1,0)$ and $\mathbf{e}_2=(0,1)$ the normal vectors of the canonical basis, we get that 
\[
\Pi_{\mathbf{e}_1}(x_1,x_2)=(0,x_2),\qquad \Pi_{\mathbf{e}_2}(x_1,x_2)=(x_1,0),\qquad \mbox{ for every } (x_1,x_2)\in\mathbb{R}^2.
\]
For $m\in\mathbb{N}\setminus\{0\}$, we will indicate by $\mathcal{H}^m$ the $m-$dimensional Hausdorff measure.
\par
Finally, for $u\in L^1_{\rm loc}(\mathbb{R}^N)$ and a measurable set $E\subseteq\mathbb{R}^N$ with positive measure, we set
\[
\mathrm{av}(u;E):=\fint_E u\,dx=\frac{1}{|E|}\,\int_E u\,dx,
\]
the integral average of $u$ over $E$.
\subsection{Fatness of the complement of a multiply connected set}
As explained in the introduction, the following geometric result will be a crucial ingredient of our main result. 
\begin{lemma}[Taylor's fatness lemma]
\label{lm:rettangolo distante}
	Let $k\in\mathbb{N}\setminus\{0\}$ and let $\Omega\subseteq\mathbb{R}^2$ be an open multiply connected set of order $k$, with finite inradius. Let $\mathcal{Q}$ be an open square with side length $10\,(\lfloor\sqrt{k}\rfloor+1)\,r_\Omega$, whose
sides are parallel to the coordinate axes.	
Then there exists a compact set $\Sigma\subseteq\overline{\mathcal{Q}}\setminus\Omega$ such that 
\begin{equation}
\label{fatprojection}
\max\Big\{\mathcal{H}^1(\Pi_{\mathbf{e}_1}(\Sigma)),\, \mathcal{H}^1(\Pi_{\mathbf{e}_2}(\Sigma))\Big\}\ge \frac{\sqrt{k}}{4}\,r_\Omega.
\end{equation}
\end{lemma}
\begin{proof}
Let us set $\delta=\lfloor\sqrt{k}\rfloor+1$, for notational simplicity.
	By dilating and traslating, there is no loss of generality in assuming $r_\Omega=1$ and
	\[
	\mathcal{Q}=Q_{5\,\delta}(0)=\left(-5\,\delta,5\,\delta\right)\times \left(-5\,\delta,5\,\delta\right).
	\]
We can suppose that $\mathcal{Q}\cap \Omega\not=\emptyset$, otherwise the proof is trivial: it would be sufficient to take $\Sigma=\overline{\mathcal{Q}}$ to get the desired conclusion.
\par	
	We then fix the following set of $4\,\delta^2$ centers
	\[
	P_{j,m}=\left(-5\,\delta +\frac{5}{2}+5\,j,\, 5\,\delta-\frac{5}{2}-5\,m\right),\qquad \mbox{ for }  j,m=0,\ldots,2\,\delta-1,
	\] 
	and take accordingly the two family of squares and disks, given by
	\[
	B_\frac{3}{2}(P_{j,m})\subset Q_\frac{5}{2}(P_{j,m}),\qquad \mbox{ for } j,m=0,\ldots,2\,\delta-1.
	\]
	\begin{figure}
	\includegraphics[scale=.3]{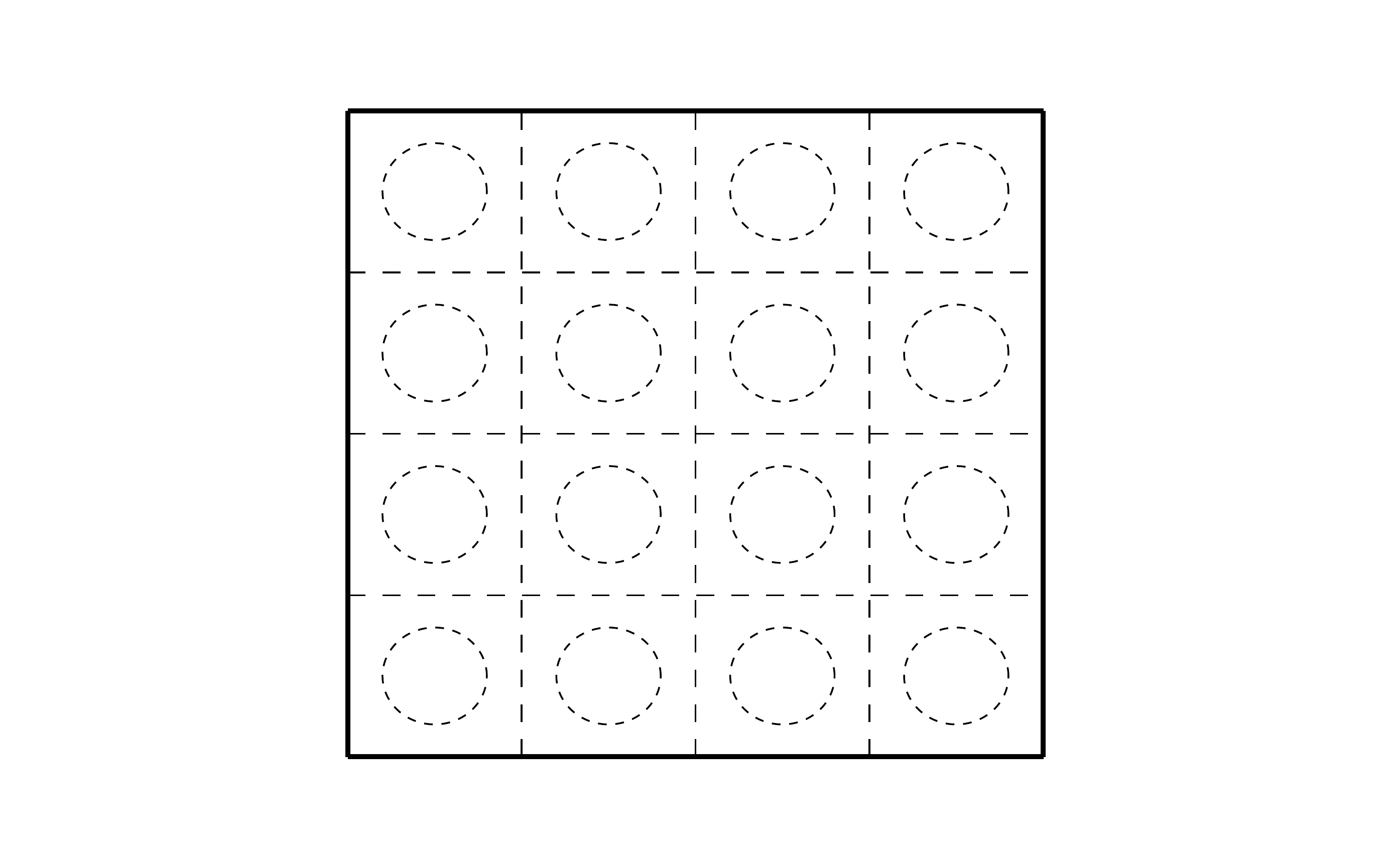}
	\caption{The construction of disks and squares in the proof of Lemma \ref{lm:rettangolo distante}, for the cases $k=1$, $k=2$ or $k=3$ (i.\,e. $\delta=2$). Each disk contains at least a point belonging to $\mathbb{R}^2\setminus \Omega$. The {\it reliable} squares are those for which such a point can be ``connected'' to the boundary of the ``cell'' containing it, with a continuum lying outside of $\Omega$.}
	\end{figure}
	\begin{figure}
\includegraphics[scale=.3]{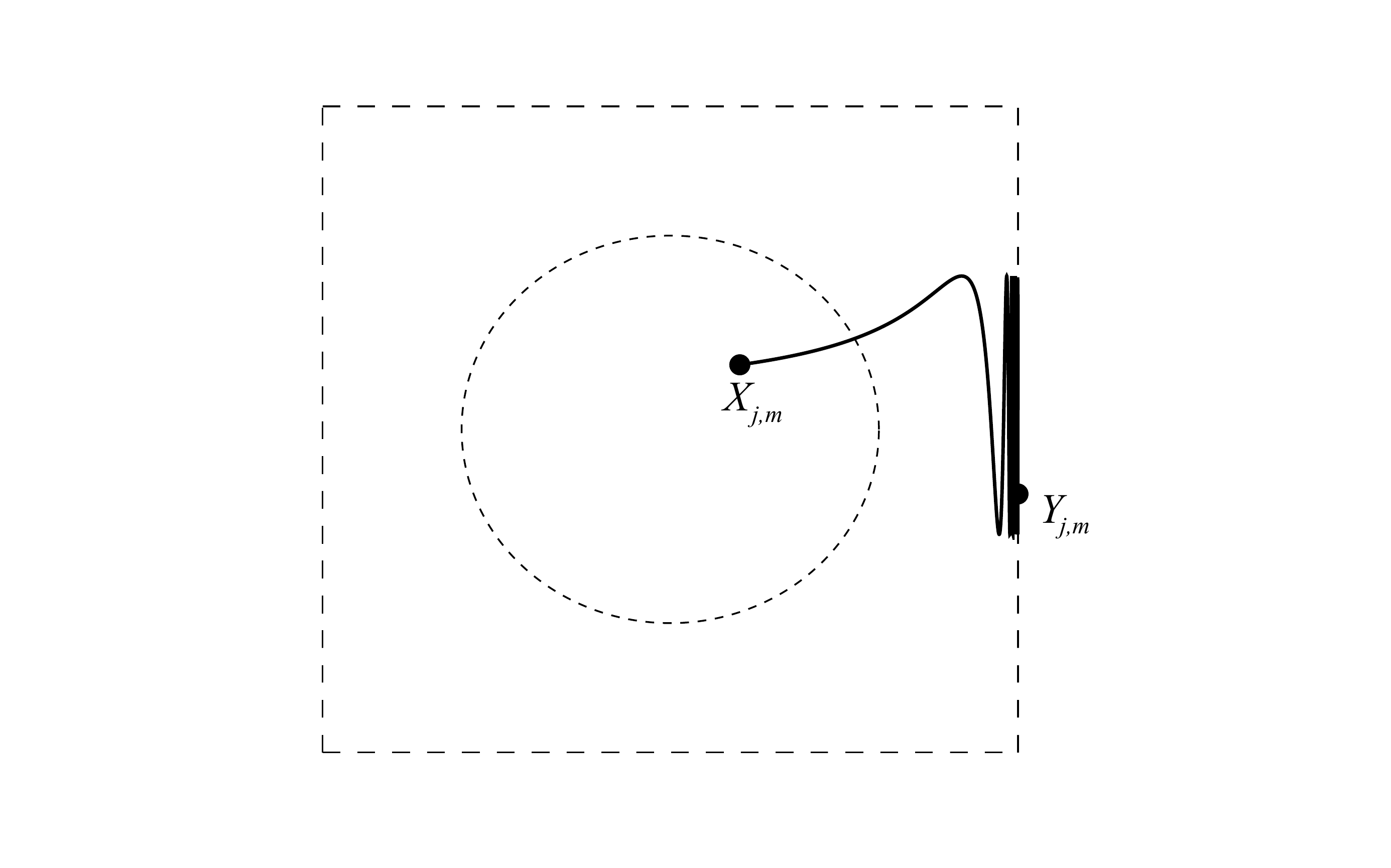}
\caption{A zooming on a reliable square $Q_{5/2}(P_{j,m})$. The bold line corresponds to a continuum which connects the point $X_{j,m}$ to the boundary of the ``cell'', lying outside of $\Omega$.}
\end{figure}
	We observe that by construction we have
\begin{equation}
\label{distanti}
\mathrm{dist}\left(B_\frac{3}{2}(P_{j,m}),\partial Q_\frac{5}{2}(P_{j,m})\right)=1,\qquad \mbox{ for every } j,m=0,\dots,2\,\delta-1
\end{equation}
Since $r_\Omega=1$, our open set $\Omega$ can not entirely contain an open disk with radius larger than $1$. Thus,
we have that each disk $B_{3/2}(P_{j,m})$ must intersect the complement $\mathbb{R}^2\setminus \Omega$. Let us select a point $X_{j,m}\in B_{3/2}(P_{j,m})\setminus\Omega$. We will say that a square $Q_{5/2}(P_{j,m})$ is:
\begin{itemize}
\item {\it unreliable} if for every continuum $K\subset \overline{Q_\frac{5}{2}(P_{j,m})}\setminus\Omega$ such that $X_{j,m}\in K$, we have 
\[
K\cap \partial Q_\frac{5}{2}(P_{j,m})=\emptyset;
\]
\vskip.2cm
\item {\it reliable} if there exists a continuum $K_{j,m}\subset \overline{Q_\frac{5}{2}(P_{j,m})}\setminus\Omega$ such that $X_{j,m}\in K_{j,m}$ and
\[
K_{j,m}\cap \partial Q_\frac{5}{2}(P_{j,m})\not=\emptyset.
\]
\end{itemize}		
We observe that every unreliable square must contain at least a connected component of $(\mathbb{R}^2)^*\setminus \Omega$. Thus, by definition of multiply connected set of order $k$, the unreliable squares can be at most $k$. Thus, if we set
\[
\mathcal{N}=\left\{(j,m)\, :\, Q_\frac{5}{2}(P_{j,m}) \mbox{ is reliable}\right\},
\]
we get\footnote{We denote by $\#$ the cardinality of a discrete set.} 
\[
\# \mathcal{N}\ge 4\,\delta^2-k=4\,\Big(\lfloor\sqrt{k}\rfloor+1\Big)^2-k\ge 3\,\Big(\lfloor\sqrt{k}\rfloor+1\Big)^2=3\,\delta^2. 
\]
That is, our square $\mathcal{Q}$ contains at least $3\,\delta^2$ reliable squares.
We want to work with these squares and their continua $K_{j,m}$ defined above. By construction, we have 
\[
K_{j,m}\subset \overline{\mathcal{Q}}\setminus\Omega.
\]
We are ready to construct the compact set $\Sigma$ of the statement: this is given by\footnote{We notice that this union is not necessarily a disjoint one.}
\[
\Sigma=\bigcup_{(j,m)\in\mathcal{N}}K_{j,m}.
\]
We need to show that its projections along the coordinate axes satisfy \eqref{fatprojection}. At this aim, we first observe that $K_{j,m}$ is a connected set, containing both the point $X_{j,m}\in B_{3/2}(P_{j,m})$ and a point $Y_{j,m}\in \partial Q_{5/2}(P_{j,m})$. 
By recalling \eqref{distanti}, we have that 
\[
|X_{j,m}-Y_{j,m}|\ge 1.
\]	
Moreover, we have that at least one of the two quantities
\[
|\Pi_{\mathbf{e}_1}(X_{j,m})-\Pi_{\mathbf{e}_1}(Y_{j,m})| \qquad \mbox{ or } \qquad |\Pi_{\mathbf{e}_2}(X_{j,m})-\Pi_{\mathbf{e}_2}(Y_{j,m})|,
\]
is larger than or equal to $1$ (recall that all the squares involved have sides parallel to the coordinate axes). By using this fact, together with the fact that both projections
\[
\Pi_{\mathbf{e}_i}\left(K_{j,m}\right),\qquad \mbox{ for } i=1,2,
\]
coincide with a segment containing both $\Pi_{\mathbf{e}_i}(X_{j,m})$ and $\Pi_{\mathbf{e}_i}(Y_{j,m})$, we can finally assure that at least one of the two projections of $K_{j,m}$ has a length larger than or equal to $1$.
In order to conclude, we need to take care of the possible overlaps in these projections. Let us denote by $J_1,J_2\in\mathbb{N}$ the following numbers
\[
J_i=\# \Big\{K_{j,m}\, :\, \mathcal{H}^1(\Pi_{\mathbf{e}_i}(K_{j,m}))\ge 1\Big\},\qquad \mbox{ for } i=1,2.
\]
According to the previous discussion, we have 
\[
J_1+J_2\ge 3\,\delta^2\qquad \mbox{ and thus in particular }\qquad \max\{J_1,J_2\}\ge \delta^2.
\]
Without loss of generality, we can suppose that $J_1\ge J_2$. This implies that there are at least $\delta^2$ ``good" projections, i.e. projections with length at least $1$, on the second coordinate axis. We need to estimate the number of such projections, modulo overlaps:
observe that for every fixed $m\in \{0,\dots,2\,\delta-1\}$, the array of squares
\[
\overline{Q_{0,m}(P_{0,m})},\dots,\overline{Q_{2\,\delta-1,m}(P_{2\,\delta-1,m})}
\]
all have the same projection. Thus the number of distinct projections is at least
\[
\frac{\delta^2}{2\,\delta}=\frac{\delta}{2}.
\]
As a technical and annoying fact, we record that this could fail to be a natural number.
However, if we set
\[
\Lambda_k=\left\{\begin{array}{cc}
1,& \mbox{ for }k\in\{1,2,3\},\\
&\\
\dfrac{\lfloor\sqrt{k}\rfloor}{2},& \mbox{ for } k\ge 4 \mbox{ such that $\lfloor\sqrt{k}\rfloor$ is even},\\
&\\
\dfrac{\lfloor\sqrt{k}\rfloor-1}{2},& \mbox{ for } k\ge 4 \mbox{ such that $\lfloor\sqrt{k}\rfloor$ is odd},
\end{array}
\right.
\]
we have 
\[
\frac{\delta}{2}\ge \Lambda_k.
\]
Thus we have at least $\Lambda_k$ projections on the first coordinate axis, each having length at least $1$. This in turn yields
\[
\mathcal{H}^1(\Pi_{\mathbf{e}_1}(\Sigma))\ge \Lambda_k.
\] 
Finally, by observing that $\Lambda_k\ge \sqrt{k}/4$, we get the claimed conclusion. 
\end{proof}

\subsection{Functional spaces}

We need some definitions from the theory of fractional Sobolev spaces. We refer the reader to \cite{DPV, EE} for a brief introduction to these spaces, as well as for further references.
\par
Let $0<s<1$ and $1<p<\infty$, for a measurable set $E\subseteq\mathbb{R}^N$ we recall the definition of Sobolev-Slobodecki\u{\i} seminorm
\[
[u]_{W^{s,p}(E)}:=\left(\iint_{E\times E}\frac{|u(x)-u(y)|^p}{|x-y|^{N+s\,p}}\,dx\,dy\right)^\frac{1}{p},\qquad \mbox{ for } u\in L^1_{\rm loc}(E).
\]
Accordingly, we consider 
\[
W^{s,p}(E)=\Big\{u\in L^p(E)\, :\, [u]_{W^{s,p}(E)}<+\infty\Big\},
\]
endowed with the norm
\[
\|u\|_{W^{s,p}(E)}=\|u\|_{L^p(E)}+[u]_{W^{s,p}(E)},\qquad \mbox{ for every } u\in W^{s,p}(E).
\]
Occasionally, we will need these definitions for $p=\infty$. For $0<s<1$, we set
\[
W^{s,\infty}(E)=\Big\{u\in L^\infty(E)\, :\, [u]_{W^{s,\infty}(E)}<+\infty\Big\},
\]
where 
\[
[u]_{W^{s,\infty}(E)}:=\sup_{x,y\in E, x\not=y} \frac{|u(x)-u(y)|}{|x-y|^s}
\]
When $E\subseteq\mathbb{R}^N$ is an open set, we will also consider the classical Sobolev space
\[
W^{1,p}(E)=\Big\{u\in L^p(E)\, :\, [u]_{W^{1,p}(E)}<+\infty\Big\},
\]
where we used the symbol
\[
[u]_{W^{1,p}(E)}:=\|\nabla u\|_{L^p(E)},\qquad \mbox{ for every } u\in W^{1,p}(E).
\]
The space $W^{1,p}(E)$ will be endowed with the norm
\[
\|u\|_{W^{1,p}(E)}=\|u\|_{L^p(E)}+[u]_{W^{1,p}(E)},\qquad \mbox{ for every } u\in W^{1,p}(E).
\]
In the case $p=\infty$, the definition of this space does not need any further precision. Finally, for $0<s\le 1$ and $1<p\le\infty$, the symbol $\widetilde W^{s,p}_0(\Omega)$ will denote the closure of $C^\infty_0(\Omega)$ in the space $W^{s,p}(\mathbb{R}^N)$. 
By $W^{s,p}_{\rm loc}(\mathbb{R}^N)$ we mean the collection of functions which are in $W^{s,p}(B_R)$, for every $R>0$.

\section{Some facts from the theory of fractional Sobolev spaces}
\label{sec:3}

Unless otherwise stated, all the results of this section are valid in every dimension $N\ge 1$.
\vskip.2cm\noindent
We start with the following generalization of \cite[Lemma 2.2]{BB}. The main focus is on the precise form of the estimates.

\begin{prop}\label{lm:extension}
	Let $r>0$ and $x_0\in\mathbb{R}^N$, there exists a linear extension operator
	\[
	\mathcal{E}_r:L^1(B_r(x_0))\to L^1_{\rm loc}(\mathbb{R}^N),
	\]
	with the following property: 
	\par
	for $0<s\le 1$ and $1< p\le\infty$ it maps $W^{s,p}(B_r(x_0))$ to $W^{s,p}_{\rm loc}(\mathbb{R}^N)$. Moreover, for every $u\in W^{s,p}(B_r(x_0))$ and every $R>r$ we have\footnote{In the case $p=\infty$, we use the convention $1/\infty=0$.}
	\begin{equation}
		\label{continuo}
		\Big[\mathcal{E}_r[u]\Big]_{W^{s,p}(B_R(x_0))}\le 4^\frac{1}{p}\,\left(\frac{R}{r}\right)^\frac{4\,N}{p}\,[u]_{W^{s,p}(B_r(x_0))},
	\end{equation}
	and
	\begin{equation}
		\label{continuo2}
		\Big\|\mathcal{E}_r[u]\Big\|_{L^p(B_R(x_0))}\le 2^\frac{1}{p}\,\left(\frac{R}{r}\right)^\frac{2N}{p}\,\|u\|_{L^p(B_r(x_0))}.
	\end{equation}
\end{prop}
\begin{proof}
We first prove the result at scale $1$, i.e. when $r=1$. Then we will show how to get the general result, by an easy scaling argument.
\vskip.2cm\noindent
{\it Case $r=1$}. For $0<s<1$ and $p=2$, this is exactly \cite[Lemma 2.2]{BB}. We also observe that the very same proof applies to the case $1< p\le \infty$, thus we omit the straightforward modifications. 
\par
We now come to the case $s=1$ and $1< p<\infty$. We take $u\in W^{1,p}(B_1(x_0))$, thus by \cite[Proposition 3.1]{EE} we have $u\in W^{s,p}(B_1(x_0))$ for every $0<s<1$, as well. 
From the previous step, we know that 
\[
(1-s)^\frac{1}{p}\,\Big[\mathcal{E}_1[u]\Big]_{W^{s,p}(B_R(x_0))}\le 4^\frac{1}{p}\,\left(\frac{R}{r}\right)^\frac{4\,N}{p}\,(1-s)^\frac{1}{p}\,[u]_{W^{s,p}(B_r(x_0))}.
\]
By using \cite[Theorem 2]{BBM}, we get the desired result by taking the limit as $s$ goes to $1$, that is
\[
\Big[\mathcal{E}_1[u]\Big]_{W^{1,p}(B_R(x_0))}\le 4^\frac{1}{p}\,\left(\frac{R}{r}\right)^\frac{4\,N}{p}\,[u]_{W^{1,p}(B_r(x_0))}.
\]
Finally, the case $p=\infty$ can be obtained from the last formula in display, by taking the limit as $p$ goes to $\infty$.
\vskip.2cm\noindent
{\it Case $r\not =1$}. At first, we need a notation. For every $\tau>0$, we indicate by 
\[
\mathcal{T}_\tau(x)=\tau\,(x-x_0)+x_0,\qquad \mbox{ for every } x\in\mathbb{R}^N.
\]
Then the operator $\mathcal{E}_r$ can be simply defined as 
\[
\mathcal{E}_r[u]:=(\mathcal{E}_1[u\circ \mathcal{T}_r])\circ \mathcal{T}_\frac{1}{r}.
\]
In other words, given a function $u\in L^1(B_r(x_0))$, we first scale it to a function defined on $B_1(x_0)$, then extend it with $\mathcal{E}_1$ and finally scale back this extension. Observe that for $x\in B_r(x_0)$, we have
\[
\mathcal{E}_r[u](x)=\mathcal{E}_1[u\circ \mathcal{T}_r]\left(\frac{x-x_0}{r}+x_0\right)=u\left(\mathcal{T}_r\left(\frac{x-x_0}{r}+x_0\right)\right)=u(x).
\]
By using the scaling properties of the norms involved, it is easy to see that this operator has the desired properties.
\end{proof}
By combining Proposition \ref{lm:extension} with Lemma \ref{lemma: cambio di variabili bi-lipschitz} in Appendix \ref{app:A}, we can get a universal linear extension operator for any $K\subset \mathbb{R}^N$ open bounded convex set. The control on the relevant constants is quite precise and useful for our scopes. In what follows, for every  $x_0\in K$, we introduce the following geometric quantities 
\[
d_K(x_0)=\min_{x\in\partial K} |x-x_0|,\qquad D_K(x_0)=\max_{x\in \partial K} |x-x_0|.
\]

\begin{cor}
\label{cor:extensionK}
Let $K\subseteq\mathbb{R}^N$ be an open bounded convex set and $x_0\in K$, there exists a linear extension operator 
	\[
	\mathcal{E}_K:L^1(K)\to L^1_{\rm loc}(\mathbb{R}^N),
	\]
	with the following property: 
	\par
	for $0<s\le 1$ and $1< p\le\infty$ it maps $W^{s,p}(K)$ to $W^{s,p}_{\rm loc}(\mathbb{R}^N)$. Moreover, for every $u\in W^{s,p}(K)$ and every $R>1$ we have
		\begin{equation}
		\label{continuoK}
	\Big[\mathcal{E}_K(u)\Big]_{W^{s,p}(K_R(x_0))}\le \big(4\cdot 6^{3\,N+s\,p}\big)^\frac{1}{p}\,R^\frac{4\,N}{p}\,\left(\frac{D_K(x_0)}{d_K(x_0)}\right)^{\frac{6\,N}{p}+2\,s}\,[u]_{W^{s,p}(K)},
	\end{equation} 
	and
	\begin{equation}\label{continuo1K}
	\|\mathcal{E}_K(u)\|_{L^p(K_R(x_0))}\le \big(2\cdot 6^N)^\frac{1}{p}\,R^\frac{2\,N}{p}\,\left(\frac{D_K(x_0)}{d_K(x_0)}\right)^\frac{2\,N}{p}\,\|u\|_{L^p(K)},
	\end{equation}
	where 
	\[
	K_R(x_0):=R\,(K-x_0)+x_0=\Big\{R\,(x-x_0)+x_0\,:\,x\in K\Big\}.
	\]
\end{cor}
\begin{proof}
The operator $\mathcal{E}_K$ is constructed as follows: by indicating with $\Phi_{K,x_0}:\mathbb{R}^N\to\mathbb{R}^N$ the bi-Lipschitz homeomorphism of Lemma \ref{lemma: cambio di variabili bi-lipschitz}, for every $u\in L^1_{\rm loc}(K)$, we define
\[
\mathcal{E}_K[u]:=\Big(\mathcal{E}_1[u\circ \Phi^{-1}_{K,x_0}]\Big)\circ \Phi_{K,x_0},
\]
where $\mathcal{E}_1$ is the operator of Proposition \ref{lm:extension}. In other words, we transplant $u$ to the unit ball centered at $x_0$, then we extend this function to the whole $\mathbb{R}^N$ by means of $\mathcal{E}_1$ and finally compose the resulting function with $\Phi_{K,x_0}$. 
\par
By construction, it is clear that $\mathcal{E}_K$ is linear and such that 
\[
\mathcal{E}_K[u](x)=u(x),\qquad \mbox{ for } x\in K.
\]
The continuity estimates \eqref{continuoK} and \eqref{continuo1K} can now be proved from the corresponding estimates for $\mathcal{E}_1$, by using the properties of $\Phi_{K,x_0}$ and $\Phi_{K,x_0}^{-1}$, encoded by Lemma \ref{lemma: cambio di variabili bi-lipschitz}. We leave the details to the reader.
\end{proof}

In what follows, given a ball $B_r(x_0)\subseteq\mathbb{R}^N$, a point $x\in B_r(x_0)$ and a direction $\omega\in\mathbb{S}^{N-1}$, we set
\[
R_\omega(x)=\sup\Big\{\varrho\in\mathbb{R}\, :\, x+\varrho\,\omega\in B_r(x_0)\Big\},
\]
and 
\[
r_\omega(x)=\inf\Big\{\varrho\in\mathbb{R}\, :\, x+\varrho\,\omega\in B_r(x_0)\Big\},
\]
see Figure \ref{fig:decomposition}. The following result is interesting in itself.
\begin{figure}
\includegraphics[scale=.3]{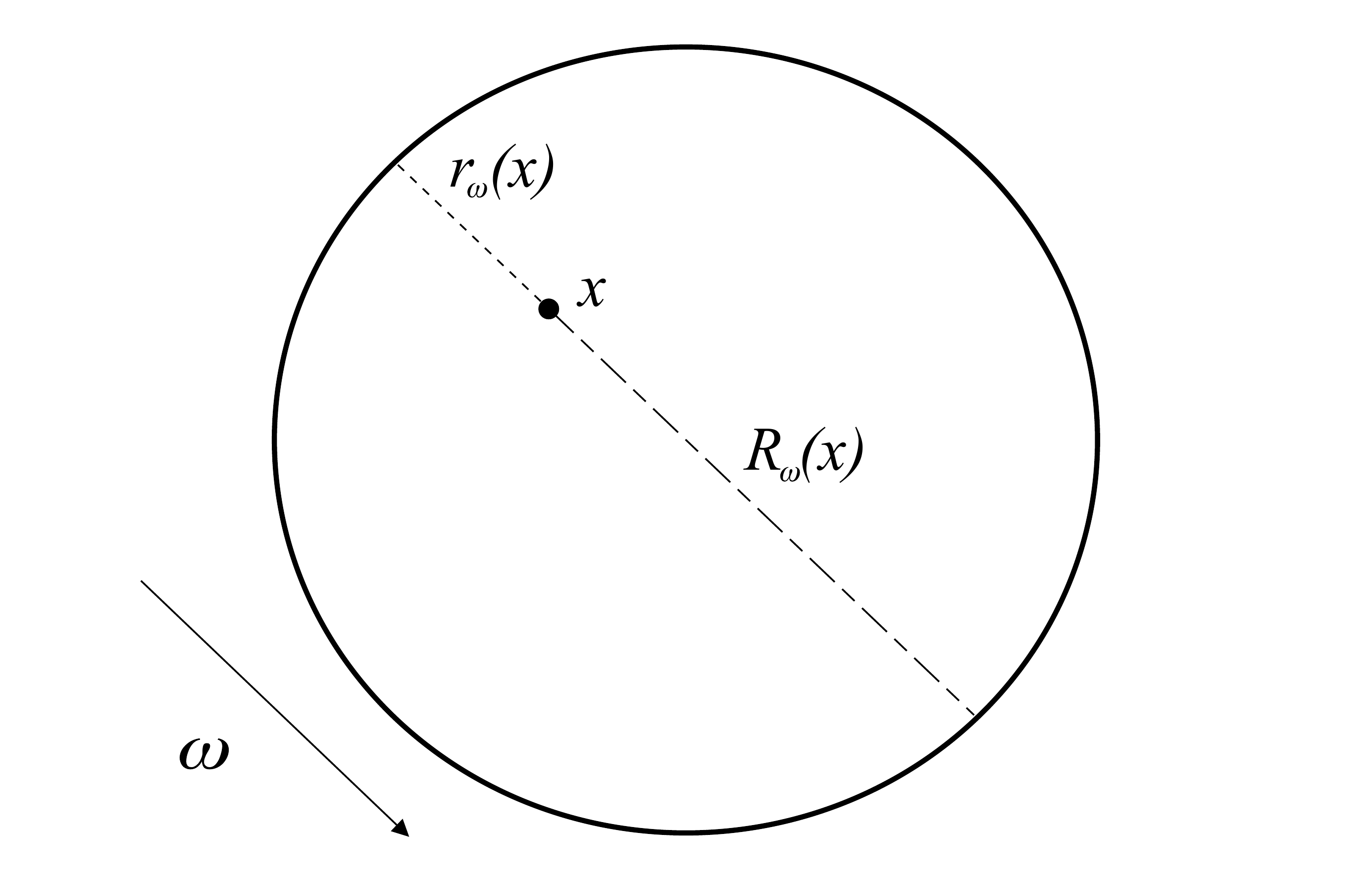}
\caption{The two quantities $R_\omega(x)$ and $r_\omega(x)$.}
\label{fig:decomposition}
\end{figure}
\begin{prop}[Directional fractional derivatives]\label{prop: derivata direzionale}
\label{prop:direzioni}
	Let $0<s<1$ and $r>0$, for every $u\in C^1(\overline{B_r(x_0)})$ and every $\omega\in\mathbb{S}^{N-1}$, we have 
	\begin{equation}
		\label{francesca}
		\int_{B_r(x_0)}\left(\int_{r_\omega(x)}^{R_\omega(x)} \frac{|u(x)-u(x+\varrho\,\omega)|^2}{|\varrho|^{1+2\,s}}\,d\varrho\right)\,dx\le \mathcal{A}\,[u]^2_{W^{s,2}(B_r(x_0))},
	\end{equation}
	for some $\mathcal{A}=\mathcal{A}(N)>0$.
\end{prop}
\begin{proof}
	Without loss of generality, we can assume that $x_0$ coincides with the origin. We use Proposition \ref{lm:extension} and estimate \eqref{continuo} with $R=4\,r$, so to get
	\begin{equation}
		\label{0}
		\begin{split}
			\iint_{B_r\times B_r} \frac{|u(x)-u(y)|^2}{|x-y|^{N+2\,s}}\,dx\,dy&\ge \frac{1}{C}\,\iint_{B_{4r}\times B_{4r}} \frac{|\mathcal{E}_r[u](x)-\mathcal{E}_r[u](y)|^2}{|x-y|^{N+2\,s}}\,dx\,dy\\
			&\ge \frac{1}{C}\,\int_{B_{r}}\left(\int_{B_{2r}(x)} \frac{|\mathcal{E}_r[u](x)-\mathcal{E}_r[u](y)|^2}{|x-y|^{N+2\,s}}\,dy\right)\,dx\\
			&=\frac{1}{C}\,\iint_{B_{r}\times B_{2r}} \frac{|\mathcal{E}_r[u](x)-\mathcal{E}_r [u](x+h)|^2}{|h|^{N+2\,s}}\,dx\,dh,
		\end{split}
	\end{equation}
	where $C$ only depends on the dimension $N$.
	In the last identity, we used the change of variable $y=x+h$. 
	\par
	From now on, we will write $\widetilde{u}$ in place of $\mathcal{E}_r[u]$, for notational simplicity. We then introduce the following {\it $K-$functional}
	\begin{equation}\label{eq: def K}
	\mathcal{K}(t,u)=\inf_{v\in W^{1,2}(B_r)} \Big[\|u-v\|_{L^2(B_{r})}+t\,[\nabla v]_{W^{1,2}(B_{r})}\Big],\qquad \mbox{ for } t\in[0,2r].
	\end{equation}
	We claim that the following two estimates hold: there exist two constants $A_1,A_2>0$ depending on the dimension $N$ only, such that
	\begin{equation}
		\label{1}
		\int_0^{2r} \left(\frac{\mathcal{K}(t,u)}{t^s}\right)^2\,\frac{dt}{t}\le A_1\, \iint_{B_r\times B_{2r}} \frac{|\widetilde{u}(x)-\widetilde{u}(x+h)|^2}{|h|^{N+2\,s}}\,dx\,dh,
	\end{equation}
	and 
	\begin{equation}
		\label{2}
		\int_{B_r}\left(\int_{-2r}^{2r} \frac{|\widetilde{u}(x)-\widetilde{u}(x+\varrho\,\omega)|^2}{|\varrho|^{1+2\,s}}\,d\varrho\right)\,dx\le A_2\, \int_0^{2r} \left(\frac{\mathcal{K}(t,u)}{t^s}\right)^2\,\frac{dt}{t},\quad \mbox{ for every }\omega\in \mathbb{S}^{N-1}.
	\end{equation}
	Observe that by joining \eqref{0}, \eqref{1} and \eqref{2}, we would get 
	\[
	\begin{split}
		\iint_{B_r\times B_r} \frac{|u(x)-u(y)|^2}{|x-y|^{N+2\,s}}\,dx\,dy&\ge \frac{1}{C\cdot A_1\cdot A_2}\, \int_{B_r}\left(\int_{-2r}^{2r} \frac{|\widetilde{u}(x)-\widetilde{u}(x+\varrho\,\omega)|^2}{|\varrho|^{1+2\,s}}\,d\varrho\right)\,dx,
	\end{split}
	\]
	and thus the desired conclusion \eqref{francesca} would follow, once observed that $R_\omega(x)\le 2\,r$ and $r_\omega(x)\ge -2\,r$, together with the fact that $\widetilde{u}=u$ on $B_r$. Thus we are left with establishing the validity of both \eqref{1}  and \eqref{2}.
	\vskip.2cm\noindent
	In order to prove \eqref{1}, we proceed exactly as in the proof of \cite[Proposition 4.5]{BS}, up to some necessary modifications.
	At first, it is useful to define
	\[
	U(h)=\left(\int_{B_r} |\widetilde{u}(x+h)-\widetilde{u}(x)|^2\,dx\right)^\frac{1}{2},\qquad h\in B_{2r}.
	\]
	Thus, by definition, the right-hand side of \eqref{1} can be rewritten as
	\[
	\iint_{B_r\times B_{2r}} \frac{|\widetilde{u}(x)-\widetilde{u}(x+h)|^2}{|h|^{N+2\,s}}\,dx\,dh=\int_{B_{2r}} \frac{U(h)^2}{|h|^{N+2\,s}}\,dh.
	\]
	We also define
	\[
	\overline{U}(\varrho)=\fint_{\partial B_\varrho} U\,d\mathcal{H}^{N-1},\qquad \mbox{ for }0<\varrho\le 2r.
	\]
	By Jensen's inequality we obtain
	\begin{equation}
		\label{stimajensen}
		\begin{split}
			\int_0^{2r} \overline U^2\,\frac{d\varrho}{\varrho^{1+2\,s}}&\le \frac{1}{N\,\omega_N}\,\int_{0}^{2r} \left(\int_{\partial B_\varrho} U^2\,d\mathcal{H}^{N-1}\right)\,\frac{d\varrho}{\varrho^{N+2\,s}}\\
			&= \frac{1}{N\,\omega_N}\,\int_{B_{2r}} \frac{U(h)^2}{|h|^{N+2\,s}}\,dh=\frac{1}{N\,\omega_N}\,\iint_{B_r\times B_{2r}} \frac{|\widetilde{u}(x)-\widetilde{u}(x+h)|^2}{|h|^{N+2\,s}}\,dx\,dh.
		\end{split}
	\end{equation}
	We now take the compactly supported Lipschitz function 
	\[
	\psi(x)=\frac{N+1}{\omega_N}\,(1-|x|)_+,
	\] 
	where $(\,\cdot\,)_+$ stands for the positive part.
	Observe that $\psi$ has unit $L^1$ norm, by construction. We then define the rescaled function
	\[
	\psi_t(x)=\frac{1}{t^N} \,\psi\left(\frac{x}{t}\right),\qquad \mbox{ for }0<t\le 2r,
	\]
	which is supported on $\overline{B_t}$.
	By observing that $\psi_t\ast \widetilde{u}\in W^{1,2}(B_r)$, from the definition of $\mathcal{K}(t,u)$ we have
	\[
	\mathcal{K}(t,u)\le \|u-\psi_t\ast \widetilde{u}\|_{L^2(B_{r})}+t\,[\nabla \psi_t\ast \widetilde{u}]_{W^{1,2}(B_{r})}.
	\]
	We estimate the two norms in the right-hand side separately: for the first one, by Minkowski's inequality and Fubini's Theorem we obtain
	\[
	\begin{split}
		\|u-\psi_t\ast \widetilde{u}\|_{L^2(B_{r})}&=\left\|\int_{B_t} [\widetilde{u}(\cdot)-\widetilde{u}(\cdot-y)]\,\psi_t(y)\,dy\right\|_{L^2(B_{r})}\\
		&\le \int_{B_t}\left(\int_{B_{r}} |\widetilde{u}(x)-\widetilde{u}(x-y)|^2\,dx\right)^\frac{1}{2}\,\psi_t(y)\,dy\\
		&=\int_{B_t} U(-y)\,\psi_{t}(y)\,dy\le \frac{N+1}{\omega_N\,t^N}\,\int_{B_t} U(-y)\,dy\\
		&=\frac{N\,(N+1)}{t^N}\,\int_0^t \overline{U}\,\varrho^{N-1}\,d\varrho\le\frac{N\,(N+1)}{t}\,\int_0^t \overline{U}\,d\varrho.
	\end{split}
	\]
	In the first identity we used that $\widetilde{u}=u$ in $B_r$, in the last inequality we used that $\varrho^{N-1}\le t^{N-1}$.
	For the second norm, we first observe that the Divergence Theorem gives
	\[
	\int_{B_t} \nabla \psi_t(y)\,dy=0,
	\]
	thus we can write
	\[
	\nabla \psi_t\ast \widetilde{u}=(\nabla \psi_t)\ast \widetilde{u}=\int_{B_t} \nabla \psi_t(y)\,[\widetilde{u}(x-y)-\widetilde{u}(x)]\,dy.
	\]
	Thus, again Minkowski's inequality yields
	\[
	\begin{split}
		[\nabla \psi_t\ast \widetilde{u}]_{W^{1,2}(B_r)}&=\left\|\int_{B_t} \nabla \psi_t(y)\,[\widetilde{u}(\cdot-y)-\widetilde{u}(\cdot)]\,dy\right\|_{L^2(B_r)}\\
		&\le \int_{B_t}\left( \int_{B_r} |\widetilde{u}(x-y)-\widetilde{u}(x)|^2\,dx\right)^\frac{1}{2}\,|\nabla \psi_t(y)|\,dy\\
		&\le \frac{N+1}{\omega_N\,t^{N+1}} \,\int_{B_t} U(-y)\,dy\le \frac{N\,(N+1)}{t^2}\,\int_0^t \overline{U}\,d\varrho.
	\end{split}
	\]
	In conclusion, we have obtained
	\begin{equation}
		\label{stimaK}
		\mathcal{K}(t,u)\le \frac{2\,N\,(N+1)}{t}\,\int_0^t \overline{U}\,d\varrho,\qquad \mbox{ for every } 0<t\le 2\,r.
	\end{equation}
	By raising to the power $2$, dividing by $t^{2\,s+1}$ and integrating, the previous estimate yields
	\[
	\int_0^{2r} \left(\frac{\mathcal{K}(t,u)}{t^s}\right)^2\,\frac{dt}{t}\le \Big(2\,N\,(N+1)\Big)^2\,\int_0^{2r} \left(\frac{1}{t}\,\int_0^t \overline{U}\,d\varrho\right)^2\,\frac{dt}{t^{1+2\,s}}.
	\]
	If we now use the one-dimensional Hardy inequality (see \cite[Teorema 1]{GTa}) for the function $t\mapsto \int_0^t \overline{U}\,d\varrho$,
	we get
	\[
	\begin{split}
		\int_0^{2r} \left(\frac{\mathcal{K}(t,u)}{t^s}\right)^2\,\frac{dt}{t}&\le \left(\frac{2\,N\,(N+1)}{s+1}\right)^2\,\int_0^{2\,r} \overline{U}^2\frac{dt}{t^{1+2\,s}}\\
		&\le  \frac{4\,N\,(N+1)^2}{\omega_N}\,\iint_{B_r\times B_{2r}} \frac{|\widetilde{u}(x)-\widetilde{u}(x+h)|^2}{|h|^{N+2\,s}}\,dx\,dh,
	\end{split}
	\]
	where we used \eqref{stimajensen} in the second inequality. This proves \eqref{1}, as desired.
	\vskip.2cm\noindent
	The proof of \eqref{2} is similar to that of \cite[Proposition B.1]{BCV}, but some technical modifications are needed, here as well. We take $0<|\varrho|\le r$, by definition of the $K-$functional there exists $v_{\varrho}\in W^{1,2}(\overline{B_r})$ such that
	\begin{equation}
		\label{dai}
		\|u-v_{\varrho}\|_{L^2(B_{r})}+\frac{|\varrho|}{2}\,\|\nabla v_{\varrho}\|_{L^2(B_{r})}\le 2\,\mathcal{K}\left(\frac{|\varrho|}{2},u\right).
	\end{equation}
	For notational simplicity, we simply write $v$ in place of $v_{\varrho}$. We also denote by $\widetilde{v}$ the extension of $v$ given by $\mathcal{E}_r[v]$.
	For $\omega\in\mathbb{S}^{N-1}$ and $|\varrho|\le 2r$, we get\footnote{In the second inequality, we use that for every $\omega\in\mathbb{S}^{N-1}$ and every $|\varrho|\le 2r$, we have
		\[
		\left(\int_{B_r} |\widetilde{v}(x+\varrho\,\omega)-\widetilde{v}(x)|^2\,dx\right)^\frac{1}{2}\le |\varrho|\,\left(\int_{B_{3r}} |\partial_\omega \widetilde{v}|^2\,dx\right)^\frac{1}{2}.
		\]
	}
	\[
	\begin{split}
		\left(\int_{B_r} |\widetilde{u}(x+\varrho\,\omega)-\widetilde{u}(x)|^2\,dx\right)^\frac{1}{2}&\le \left(\int_{B_r} |\widetilde{u}(x+\varrho\,\omega)-\widetilde{v}(x+\varrho\,\omega)-\widetilde{u}(x)+\widetilde{v}(x)|^2\,dx\right)^\frac{1}{2}\\
		&+\left(\int_{B_r} |\widetilde{v}(x+\varrho\,\omega)-\widetilde{v}(x)|^2\,dx\right)^\frac{1}{2}\\
		&\le 2\,\|\widetilde{u}-\widetilde{v}\|_{L^2(B_{3r})}+|\varrho|\,\|\partial_\omega \widetilde{v}\|_{L^2(B_{3r})}\\
		&\le 2\,\left(\|\widetilde{u}-\widetilde{v}\|_{L^2(B_{3r})}+\frac{|\varrho|}{2}\,\|\nabla \widetilde{v}\|_{L^2(B_{3r})}\right).
	\end{split}
	\]
	In the last estimate, we used the pointwise inequality $|\partial_\omega \widetilde{v}|\le |\nabla \widetilde{v}|$. We can now use the properties of our extension operator $\mathcal{E}_r$, in order to replace the norms over $B_{3r}$ with those on $B_r$. By Proposition \ref{lm:extension}, we have
	\[
	\|\widetilde{u}-\widetilde{v}\|_{L^2(B_{3r})}=\|\mathcal{E}_r[u]-\mathcal{E}_r[v]\|_{L^2(B_{3r})}=\|\mathcal{E}_r[u-v]\|_{L^2(B_{3r})}\le \sqrt{2}\cdot 3^{N}\,\|u-v\|_{L^2(B_r)},
	\]
	and also
	\[
	\|\nabla \widetilde{v}\|_{L^2(B_{3r})}=\Big[\nabla \mathcal{E}_r[v]\Big]_{W^{1,2}(B_{3r})}\le 2\cdot 9^{N}\,[v]_{W^{1,2}(B_r)}.
	\]
	This leads to 
	\[
	\left(\int_{B_r} |u(x+\varrho\,\omega)-u(x)|^2\,dx\right)^\frac{1}{2}\le C\,\left(\|u-v\|_{L^2(B_r)}+\frac{|\varrho|}{2}\,[\nabla v]_{W^{1,2}(B_r)}\right).
	\]
	By combining this estimate with \eqref{dai}, we then obtain for $0<|\varrho|\le 2r$
	\[
	\int_{B_r} \frac{|\widetilde{u}(x+\varrho\,\omega)-\widetilde{u}(x)|^2}{|\varrho|^{1+2\,s}}\,dx\le 4\,C^2\,|\varrho|^{-1-2\,s}\,\mathcal{K}\left(\frac{|\varrho|}{2},u\right)^2.
	\]
	We now integrate with respect to $\varrho$ and make a change of variable. This yields \eqref{2}, as desired. The proof is now over.
\end{proof}
As a straightforward consequence of Proposition \ref{prop:direzioni}, we also get the following result (see also \cite[Lemma A.4]{BD}).
\begin{cor}
\label{coro:francesca}
	Let $0<s<1$, for every $u\in C^\infty_0(\mathbb{R}^N)$ and every $\omega\in\mathbb{S}^{N-1}$, we have 
	\begin{equation}
		\label{francescabis}
		\int_{\mathbb{R}^N}\left(\int_\mathbb{R} \frac{|u(x)-u(x+\varrho\,\omega)|^2}{|\varrho|^{1+2\,s}}\,d\varrho\right)\,dx\le \mathcal{A}\,[u]^2_{W^{s,2}(\mathbb{R}^N)},
	\end{equation}
	for the same constant $\mathcal{A}=\mathcal{A}(N)>0$ appearing in \eqref{francesca}.
\end{cor}
The next result can be found in \cite{Ma} and \cite[Corollary 1]{Po}. In the latter, the estimate is slightly worse in its dependence on $s$, while in the former the result is not explicitly stated, but it must be extrapolated from the proof of \cite[Corollary 1, page 524]{Ma}. For these reasons, we prefer to provide a full proof, which in any case is different from those of the aforementioned references.
\begin{lemma}[Fractional Poincar\'e-Wirtinger inequality]
\label{lm:PW}
Let $0<s<1$, for every $u\in C^1(\overline{B_r(x_0)})$ we have 
	\[
			\Big\|u-\mathrm{av}(u;B_r(x_0))\Big\|^2_{L^2(B_r(x_0))}\le \mathcal{M}\,(1-s)\,r^{2\,s}\,[u]_{W^{s,2}(B_r(x_0))}^2,
	\]
	for some $\mathcal{M}=\mathcal{M}(N)>0$.
\end{lemma}
\begin{proof}
We can suppose that $x_0=0$, without loss of generality. We use real interpolation techniques, as in the previous result. By combining \eqref{0} and \eqref{1}, we have 
\begin{equation}
\label{ciaovladimiro!}
[u]_{W^{s,2}(B_r)}\ge \frac{1}{C}\,\int_0^{2\,r} \left(\frac{\mathcal{K}(t,u)}{t^s}\right)^2\,\frac{dt}{t},
\end{equation}
where $C$ depends on the dimension $N$ only and $\mathcal{K}(t,u)$ is still defined by \eqref{eq: def K}. We now take $0<t\le 2\,r$ and $v\in W^{1,2}(B_r)$, by the triangle inequality we get
\[
\begin{split}
t\,\|u-\mathrm{av}(u;B_r)\|_{L^2(B_r)}&\le t\,\|u-v\|_{L^2(B_r)}+t\,\|v-\mathrm{av}(v;B_r)\|_{L^2(B_r)}+t\,\|\mathrm{av}(v;B_r)-\mathrm{av}(u;B_r)\|_{L^2(B_r)}\\
&\le 2\,r\,\left(\|u-v\|_{L^2(B_r)}+\|\mathrm{av}(v;B_r)-\mathrm{av}(u;B_r)\|_{L^2(B_r)}\right)+t\,\|v-\mathrm{av}(v;B_r)\|_{L^2(B_r)}.
\end{split}
\]
By using Jensen's inequality we have 
\[
\|\mathrm{av}(v;B_r)-\mathrm{av}(u;B_r)\|_{L^2(B_r)}\le \|u-v\|_{L^2(B_r)},
\]
while by using the classical Poincar\'e-Wirtinger inequality we have
\[
\|v-\mathrm{av}(v;B_r)\|_{L^2(B_r)}\le \frac{r}{\mu}\,[v]_{W^{1,2}(B_r)},
\]
for some $\mu=\mu(N)>0$.
By keeping all these estimates together, we obtain
\[
\begin{split}
t\,\|u-\mathrm{av}(u;B_r)\|^2_{L^2(B_r)}&\le 4\,r\,\|u-v\|_{L^2(B_r)}+\frac{t\,r}{\mu(B_1)}\,[v]_{W^{1,2}(B_r)}\\
&\le C\,r\,\left(\|u-v\|_{L^2(B_r)}+t\,[v]_{W^{1,2}(B_r)}\right),
\end{split}
\]
where $C=\max\{4,1/\mu\}$ depends on $N$ only. If we now take the infimum over $v\in W^{1,2}(B_r)$, we get
\[
t\,\|u-\mathrm{av}(u;B_r)\|_{L^2(B_r)}\le C\,r\,\mathcal{K}(t,u),\qquad \mbox{ for } 0<t\le 2\,r.
\]
By raising to the power $2$, dividing by $t^{2\,s+1}$  and integrating over $(0,2\,r)$, this yields
\[
\|u-\mathrm{av}(u;B_r)\|^2_{L^2(B_r)}\,\frac{(2\,r)^{2-2\,s}}{2\,(1-s)}\le C^2\,r^2\,\int_0^{2\,r} \left(\frac{\mathcal{K}(t,u)}{t^s}\right)^2\,\frac{dt}{t}.
\]
By using this estimate in \eqref{ciaovladimiro!}, we finally get the desired conclusion.
\end{proof}
We conclude this section with a particular case of the well-known fractional Morrey--type embedding in the space of continuous functions (see for example \cite[Corollary 7.9.4]{Ho}). For our scopes, we need a precise ``quantitative'' behaviour of the relevant constant, as $s$ goes to $1$ or $1/2$. Here we take $N=1$.
\begin{teo}[Fractional Morrey-Sobolev inequality]
\label{teo:MC}
For every $1/2<s<1$  there exists a constant $\mathfrak{m}_s>0$ depending on $s$ only, such that 
	\begin{equation}
	\label{morrey-sobolev}
	\mathfrak{m}_s\,[u]^2_{W^{s-\frac{1}{2},\infty}(\mathbb{R})}\le [u]^2_{W^{s,2}(\mathbb{R})},\qquad \mbox{ for every } u\in C^\infty_0(\mathbb{R}).
	\end{equation}
	In particular, if $a<b$ we have 
		\begin{equation}
	\label{eq: |u|^2<seminorma}
		\mathfrak{m}_s\,\|u\|_{L^\infty((a,b))}^2\le (b-a)^{2\,s-1}\,[u]^2_{W^{s,2}(\mathbb{R})},\qquad \mbox{ for every }u\in C^\infty_0((a,b)).
	\end{equation}
	Moreover, the constant $\mathfrak{m}_s$ has the following asymptotic behaviour
	\[
	\mathfrak{m}_s\sim {2\,s-1},\quad\mbox{as }s\searrow1/2,\qquad\mbox{and}\qquad \mathfrak{m}_s\sim\frac{1}{1-s},\quad\mbox{as }s\nearrow1.
	\]
\end{teo}

\begin{proof}
We first observe that \eqref{eq: |u|^2<seminorma} is an easy consequence of \eqref{morrey-sobolev}. Indeed, for every $u\in C^\infty_0((a,b))$ and every $x\in(a,b)$, by \eqref{morrey-sobolev} we would get
	\[
		|u(x)|^2=|u(x)-u(a)|^2\le \frac{1}{\mathfrak{m}_s}\,(x-a)^{2\,s-1}\,[u]^2_{W^{s,2}(\mathbb{R})}\le \frac{1}{\mathfrak{m}_s}\,(b-a)^{2\,s-1}\,[u]^2_{W^{s,2}(\mathbb{R})},
	\]
	as desired.
\par
	In order to establish \eqref{morrey-sobolev}, let us take $\varphi\in C^\infty_0(\mathbb{R})$. We indicate by $\mathcal{F}[\varphi]$ its {\it Fourier transform}, defined by
	\[
	\mathcal{F}[\varphi](\xi)=\frac{1}{\sqrt{2\,\pi}}\,\int_{\mathbb{R}}\varphi(t)\,e^{-i\,t\,\xi}\,dt,\qquad \mbox{ for } \xi\in\mathbb{R}.
	\]
From the inversion formula (see \cite[Chapter VII, Section 1]{Ho}), we can write
\[
\varphi(t)=\frac{1}{\sqrt{2\,\pi}}\,\int_{\mathbb{R}}\mathcal{F}[\varphi](\xi)\,e^{i\,t\,\xi}\,d\xi,\qquad \mbox{ for } t\in\mathbb{R}.
\]	
Thus, for every $t,\tau\in\mathbb{R}$ we get
\begin{equation}
\label{sanlars}
\begin{split}
|\varphi(t)-\varphi(\tau)|&\le \frac{1}{\sqrt{2\,\pi}}\,\int_{\mathbb{R}}\Big|\mathcal{F}[\varphi](\xi)\Big|\,|e^{i\,t\,\xi}-e^{i\,\tau\,\xi}|\,d\xi\\
&\le\frac{1}{\sqrt{2\,\pi}}\,\left(\int_{\mathbb{R}} |\xi|^{2\,s}\,\Big|\mathcal{F}[\varphi](\xi)\Big|^2\,d\xi\right)^\frac{1}{2}\, \left(\int_{\mathbb{R}} \frac{|e^{i\,t\,\xi}-e^{i\,\tau\,\xi}|^2}{|\xi|^{2\,s}}\,d\xi\right)^\frac{1}{2}.
\end{split}
\end{equation}
We now recall that by \cite[Chapter VII, Section 9]{Ho}, we have 
\[
\int_{\mathbb{R}} |\xi|^{2\,s}\,\Big|\mathcal{F}[\varphi](\xi)\Big|^2\,d\xi=2\,\pi\,A_s\,[\varphi]^2_{W^{s,2}(\mathbb{R})},
\]
with the constant $A_s$ given by 
\[
A_s=\left(\int_{\mathbb{R}} \frac{|e^{i\,t}-1|^2}{|t|^{1+2\,s}}\,dt\right)^{-1},
\]
which satisfies
\[
A_s\sim 1-s,\quad \mbox{ for } s\nearrow 1\qquad \mbox{ and }\qquad A_s\sim s\quad \mbox{ for } s\searrow 0.
\]
From \eqref{sanlars}, we obtain
\begin{equation}
\label{sveglia!}
|\varphi(t)-\varphi(\tau)|\le \sqrt{A_s}\,\left(\int_{\mathbb{R}} \frac{|e^{i\,t\,\xi}-e^{i\,\tau\,\xi}|^2}{|\xi|^{2\,s}}\,d\xi\right)^\frac{1}{2}\,[\varphi]_{W^{s,2}(\mathbb{R})}.
\end{equation}
In order to conclude, we are only left with handling the integral on the right-hand side. For every $\alpha>0$, we split this integral as follows
\[
\int_{\mathbb{R}} \frac{|e^{i\,t\,\xi}-e^{i\,\tau\,\xi}|^2}{|\xi|^{2\,s}}\,d\xi=\int_{\{|\xi|\le \alpha\}} \frac{|e^{i\,t\,\xi}-e^{i\,\tau\,\xi}|^2}{|\xi|^{2\,s}}\,d\xi+\int_{\{|\xi|>\alpha\}} \frac{|e^{i\,t\,\xi}-e^{i\,\tau\,\xi}|^2}{|\xi|^{2\,s}}\,d\xi.
\]
In order to estimate the low frequencies, we use the $1-$Lipschitz character of $\vartheta\mapsto e^{i\,\vartheta}$ to infer that 
\[
|e^{i\,t\,\xi}-e^{i\,\tau\,\xi}|\le |t-\tau|\,|\xi|.
\] 
The high frequencies are dealt with by using that 
\[
|e^{i\,t\,\xi}-e^{i\,\tau\,\xi}|\le |e^{i\,t\,\xi}|+|e^{i\,\tau\,\xi}|=2.
\]
These lead to 
\[
\begin{split}
\int_{\mathbb{R}} \frac{|e^{i\,t\,\xi}-e^{i\,\tau\,\xi}|^2}{|\xi|^{2\,s}}\,d\xi&\le 2\,|t-\tau|^2\,\int_0^{\alpha} \xi^{2-2\,s}\,d\xi+8\,\int_\alpha^{+\infty}\xi^{-2\,s}\,d\xi\\
&=\frac{2}{3-2\,s}\,|t-\tau|^2\,\alpha^{3-2\,s}+\frac{8}{2\,s-1}\,\frac{1}{\alpha^{2\,s-1}},
\end{split}
\]
which is valid for every $\alpha>0$. We can now optimize this estimate with respect to $\alpha$: indeed, the quantity on the right-hand side is minimal for\footnote{We can obviously suppose that $t\not=\tau$, otherwise there is nothing to prove.} $\alpha=\alpha_0=2/|t-\tau|$.
With such a choice, we get
\[
\int_{\mathbb{R}} \frac{|e^{i\,t\,\xi}-e^{i\,\tau\,\xi}|^2}{|\xi|^{2\,s}}\,d\xi\le 4^{2-s}\,\frac{2}{(3-2\,s)\,(2\,s-1)}\,|t-\tau|^{2\,s-1}.
\]
By inserting this estimate in \eqref{sveglia!}, we finally get \eqref{morrey-sobolev} with 
\[
\mathfrak{m}_s=\frac{(3-2\,s)\,(2\,s-1)}{2\cdot 4^{2-s}\,A_s},
\]
which has the claimed asymptotic behaviour.
	\end{proof}
\begin{oss}
We point out the reference \cite{Si}, which keeps track of the dependence on $s$ in the one-dimensional fractional Morrey estimate, as this parameter goes to the borderline situation $s=1/2$ (see \cite[Corollary 26]{Si}). However, the asymptotic behaviour detected in this reference is sub-optimal. Moreover, the asymptotic behaviour as $s$ goes to $1$ is not taken into account. For these reasons, the estimates of \cite{Si} are not suitable for our needings.
\end{oss}

\section{Basics of fractional capacity}
\label{sec:4}

We start with the definition of fractional capacity.
\begin{definizione}
	Let $\Sigma\subseteq\mathbb{R}^N$ be a compact set and let $\Omega\subseteq\mathbb{R}^N$ be an open set such that $\Sigma\Subset \Omega$. For $0<s<1$, we define  the {\it fractional capacity of $\Sigma$ of order $s$ relative to $\Omega$} as the quantity
	\[
	\widetilde{\mathrm{cap}}_s(\Sigma;\Omega)=\inf_{u\in C^\infty_0(\Omega)}\left\{[u]^2_{W^{s,2}(\mathbb{R}^N)}\, :\, u\ge 1_\Sigma\right\}.
	\]
	Here $1_\Sigma$ denotes the characteristic function of $\Sigma$.
\end{definizione}

\begin{oss}
\label{oss:lipschitz}
By standard approximation arguments based on convolutions, it is easy to see that in the definition of $\widetilde{\mathrm{cap}}_s(\Sigma;\Omega)$ we can replace $C^\infty_0(\Omega)$ with Lipschitz functions having compact support in $\Omega$. We leave the details to the reader.
\end{oss}
As a straightforward consequence of both the definition and the Morrey--type inequality, we have an explicit lower-bound for the fractional capacity of a point. As simple as it is, this will play a crucial role in our main result. 
\begin{lemma}[One-dimensional capacity of a point]
\label{cor: capacità punto risp segmento} 
	Let $1/2<s<1$ and $x_0\in (a,b)$. Then
	\[
	\widetilde{\mathrm{cap}}_s\big(\{x_0\};(a,b)\big)\ge (b-a)^{1-2\,s}\,\mathfrak{m}_s,
	\]
where $\mathfrak{m}_s$ is the same constant as in Theorem \ref{teo:MC}.
\end{lemma}
\begin{proof}
	Let us take $u\in C^\infty_0((a,b))$ such that $u(x_0)\ge1$. Hence, from \eqref{eq: |u|^2<seminorma}, we get
	\[
	1\le|u(x_0)|^2\le \frac{(b-a)^{2\,s-1}}{\mathfrak{m}_s}\,[u]^2_{W^{s,2}(\mathbb{R})}.
	\]
	The thesis follows by taking the infimum over the admissible functions $u$.
\end{proof}

\subsection{A Maz'ya--type Poincar\'e inequality}

We will need the following fractional Poincar\'e inequality for functions on a cube, which vanish in a neighborhood of a set with positive fractional capacity. This is analogous to the result of \cite[Theorem A]{Ta}, but we will follow the approach of \cite[Chapter 14]{Ma}, which is more suitable for our framework. In particular, we will not explicitly relate this result to eigenvalues with mixed boundary conditions, differently from \cite{Ta}.
\begin{figure}
\includegraphics[scale=.3]{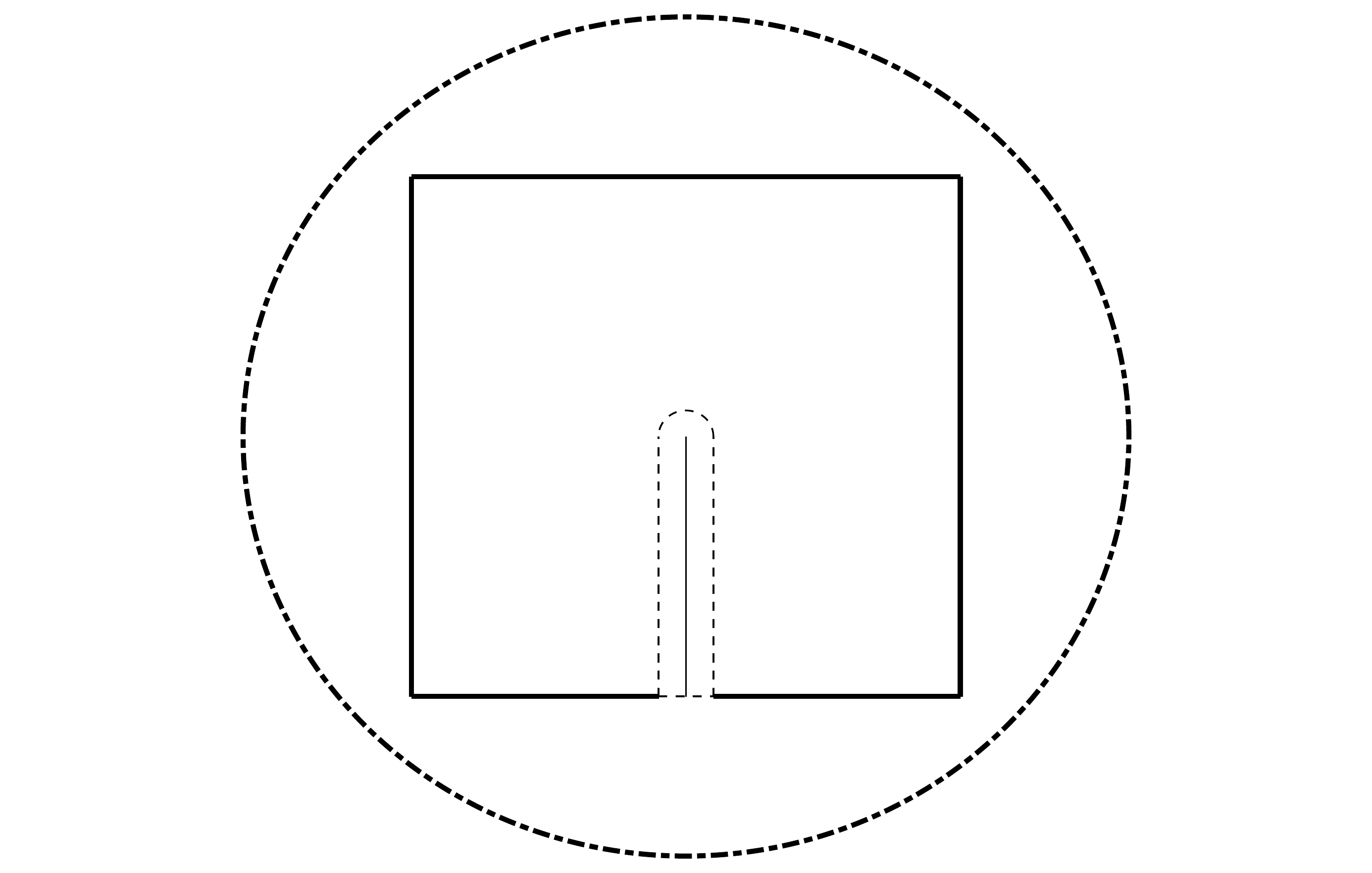}
\caption{The geometric configuration of Lemma \ref{lm:poincare_cap}: we have a smooth function defined on the square, which vanishes on the dashed neighborhood of the vertical line (i.e. the set $\Sigma$). The relative fractional capacity of $\Sigma$ is computed with respect to the surronding disk.}
\end{figure}
\begin{prop}\label{lm:poincare_cap}
	Let $0<s<1$ and let $\Sigma\subseteq \overline{Q_r(x_0)}\subseteq \mathbb{R}^N$ be a compact set. For every $R>\sqrt{N}\,r$, there exists a constant $\phi(N,R/r)>0$ such that the following Poincar\'e inequality holds
	\begin{equation}
	\label{eq: stima capacità}
	[u]^2_{W^{s,2}(Q_r(x_0))}\ge\left[ \frac{s}{r^N}\,\phi\left(N,\frac{R}{r}\right)\right]\,\widetilde{\mathrm{cap}}_s\big(\Sigma;B_R(x_0)\big)\,\|u\|^2_{L^2(Q_r(x_0))},
	\end{equation}
	for every $u\in C^\infty(\overline{Q_r(x_0)})$ with $\mathrm{dist}(\mathrm{supp}(u),\Sigma)>0$. Moreover, we have
	\[
	\lim_{t\to +\infty} \phi(N,t)=\lim_{t\searrow \sqrt{N}} \phi(N,t)=0.		
	\]
\end{prop}
\begin{proof}
The proof is lengthy, though elementary.
	Without loss of generality we can assume $x_0=0$. 
	Let $u\in C^\infty(\overline{Q_r})$ be as in the statement, we can additionally assume that 
	\begin{equation}
	\label{normalizza}
	\fint_{Q_r} |u|^2\,dx=1,
	\end{equation}
	still without loss of generality. We now use the extension operator $\mathcal{E}_K$ of Corollary \ref{cor:extensionK}, with the choices 
	\[
	K=Q_r\quad \mbox{ and }\quad x_0=0, \qquad \mbox{ so that } \frac{D_K(x_0)}{d_K(x_0)}=\sqrt{N}.
	\]
In order not to overburden the presentation, we will use the symbol $\widetilde{u}$ in place of $\mathcal{E}_K[u]$. By the properties of our extension operator, we get in particular that $\widetilde{u}$ is locally Lipschitz continuous and from \eqref{continuoK} with $p=2$ we also have
\begin{equation}
\label{stimainiziale}
[\widetilde{u}]_{W^{s,2}(B_R)}\le [\widetilde{u}]_{W^{s,2}(Q_R)}\le 
C_N\,\left(\frac{R}{r}\right)^{2\,N}\,[u]_{W^{s,2}(Q_r)}.
\end{equation}
We take a Lipschitz cut-off function $\eta$ such that
	\[
	0\le\eta\le1,\qquad \eta\equiv1\mbox{ in } \overline{B_{\sqrt{N}\,r}},\qquad \eta\equiv 0 \mbox{ in } \mathbb{R}^N\setminus B_\frac{R+\sqrt{N}\,r}{2},\qquad |\nabla \eta|\le \frac{2}{R-\sqrt{N}\,r},
	\]
	and we define $\psi=(1-\widetilde{u})\,\eta$. By recalling Remark \ref{oss:lipschitz}, we have that $\psi$ is an admissible trial function for the variational problem defining $\widetilde{\mathrm{cap}}_s(\Sigma;B_R)$. 
	By using this fact and some algebraic manipulations, we get
	\begin{equation}
\label{inizia}
	\begin{split}
\sqrt{\widetilde{\mathrm{cap}}_s(\Sigma;B_R)}&\le [\psi]_{W^{s,2}(\mathbb{R}^N)}\\
&=\left([\psi]^2_{W^{s,2}(B_R)}+2\,\int_{B_R} |\psi(x)|^2\, \left(\int_{\mathbb{R}^N\setminus B_R} \frac{dy}{|x-y|^{N+2\,s}}\right)\,dx\right)^\frac{1}{2}\\
&\le [\psi]_{W^{s,2}(B_R)}+\sqrt{2}\,\left(\int_{B_R} |\psi(x)|^2\, \left(\int_{\mathbb{R}^N\setminus B_R} \frac{dy}{|x-y|^{N+2\,s}}\right)\,dx\right)^\frac{1}{2}.
\end{split}
\end{equation}
In turn, by using the definition of $\psi$ and Minkowki's inequality, we have 
\[
\begin{split}
[\psi]_{W^{s,2}(B_R)}&\le \left(\int_{B_R}|\eta(x)|^2\left(\int_{B_R}\frac{|\widetilde{u}(x)-\widetilde{u}(y)|^2}{|x-y|^{N+2\,s}}\,dy\right)\,dx\right)^\frac{1}{2}\\
	&+\left(\int_{B_R}|1-\widetilde{u}(y)|^2\,\left(\int_{B_R}\frac{|\eta(x)-\eta(y)|^2}{|x-y|^{N+2\,s}}\,dx\right)\,dy\right)^\frac{1}{2}\\
		&\le [\widetilde{u}]_{W^{s,2}(B_R)}+\|1-\widetilde{u}\|_{L^2(B_R)}\,\sqrt{\frac{2\,C}{s\,(1-s)}}\,\|\nabla \eta\|^s_{L^\infty(B_R)}\,\|\eta\|^{1-s}_{L^\infty(B_R)},
\end{split}
\]
for some $C=C(N)>0$.
In the last inequality, we used that for every Lipschitz function $\varphi$ with compact support, we have 
\[
\sup_{x\in\mathbb{R}^N}\int_{\mathbb{R}^N} \frac{|\varphi(x)-\varphi(y)|^2}{|x-y|^{N+2\,s}}\,dy\le \frac{C}{s\,(1-s)}\,\|\nabla \varphi\|^{2\,s}_{L^\infty(\mathbb{R}^N)}\,\|\varphi\|^{2\,(1-s)}_{L^\infty(\mathbb{R}^N)},
\]
see \cite[Lemma 2.6]{BBZ}. If we now use \eqref{stimainiziale} to bound the seminorm of $\widetilde{u}$ and the properties of $\eta$, from \eqref{inizia} we get
\begin{equation}
\label{continua}
\begin{split}
\sqrt{\widetilde{\mathrm{cap}}_s(\Sigma;B_R)}&\le C_N\,\left(\frac{R}{r}\right)^{2\,N}\,[u]_{W^{s,2}(Q_r)}+\frac{2}{(R-\sqrt{N}\,r)^s}\,\sqrt{\frac{2\,C}{s\,(1-s)}}\,\|1-\widetilde{u}\|_{L^2(B_R)}\\
&+\sqrt{2}\,\left(\int_{B_R} |\psi(x)|^2\, \left(\int_{\mathbb{R}^N\setminus B_R} \frac{dy}{|x-y|^{N+2\,s}}\right)\,dx\right)^\frac{1}{2}.
\end{split}
\end{equation}
In order to handle the last term, we recall that $\psi$ identically vanishes outside $B_{(R+\sqrt{N}\,r)/2}$. Thus, we actually have 
\[
\begin{split}
\int_{B_R} |\psi(x)|^2\, \left(\int_{\mathbb{R}^N\setminus B_R} \frac{dy}{|x-y|^{N+2\,s}}\right)&=\int_{B_\frac{R+\sqrt{N}\,r}{2}} |\eta(x)|^2\,|1-\widetilde u(x)|^2 \left(\int_{\mathbb{R}^N\setminus B_R} \frac{dy}{|x-y|^{N+2\,s}}\right)\,dx\\
&\le \int_{B_\frac{R+\sqrt{N}\,r}{2}}|1-\widetilde u(x)|^2 \left(\int_{\mathbb{R}^N\setminus B_R} \frac{dy}{|x-y|^{N+2\,s}}\right)\,dx.
\end{split}
\] 
We now observe that, for every $x\in B_{(R+\sqrt{N}\,r)/2}$ and $y\not\in B_R$ we have
\[\begin{split} 
|x-y|&\ge|y|-|x|\ge |y|-\frac{R+\sqrt{N}\,r}{2}
\ge |y|-\frac{R+\sqrt{N}\,r}{2\,R}\,|y|
=\left(\frac{R-\sqrt{N}\,r}{2\,R}\right)\,|y|.
\end{split} \]
Thus, for every $x\in B_{(R+\sqrt{N}\,r)/2}$, we get
\[
\int_{\mathbb{R}^N\setminus B_R}\frac{dy}{|x-y|^{N+2\,s}}\,dy
\le \frac{N\,\omega_N}{2\,s}\,\left(\frac{2\,R}{R-\sqrt{N}\,r}\right)^{N+2\,s}\,\frac{1}{R^{2\,s}}.
\]
By collecting the previous estimates, we obtain from \eqref{continua}
\[
\begin{split}
\sqrt{\widetilde{\mathrm{cap}}_s(\Sigma;B_R)}&\le C_N\,\left(\frac{R}{r}\right)^{2\,N}\,[u]_{W^{s,2}(Q_r)}+\frac{2}{(R-\sqrt{N}\,r)^s}\,\sqrt{\frac{2\,C}{s\,(1-s)}}\,\|1-\widetilde{u}\|_{L^2(B_R)}\\
&+\sqrt{\frac{N\,\omega_N}{s}}\,\left(\frac{2\,R}{R-\sqrt{N}\,r}\right)^{\frac{N}{2}+s}\,\frac{1}{R^{s}}\,\|1-\widetilde u\|_{L^2(B_\frac{R+\sqrt{N}\,r}{2})}.
\end{split}
\]
We need to estimate the $L^2$ norm of $1-\widetilde{u}$. For this, we use the triangle inequality
	\[
	\|1-\widetilde u\|_{L^2(B_\frac{R+\sqrt{N}\,r}{2})}\le \|1-\widetilde{u}\|_{L^2(B_R)}
	\le \|1-\mathrm{av}(\widetilde u;B_R)\|_{L^2(B_R)}
	+\|\mathrm{av}(\widetilde u;B_R)-\widetilde{u}\|_{L^2(B_R)}:=\mathcal{I}_1+\mathcal{I}_2,
	\]
so that
\begin{equation}
\label{insiste}
\begin{split}
\sqrt{\widetilde{\mathrm{cap}}_s(\Sigma;B_R)}&\le C_N\,\left(\frac{R}{r}\right)^{2\,N}\,[u]_{W^{s,2}(Q_r)}+\frac{2}{(R-\sqrt{N}\,r)^s}\,\sqrt{\frac{2\,C}{s\,(1-s)}}\,\Big(\mathcal{I}_1+\mathcal{I}_2\Big)\\
&+\sqrt{\frac{N\,\omega_N}{s}}\,\left(\frac{2\,R}{R-\sqrt{N}\,r}\right)^{\frac{N}{2}+s}\,\frac{1}{R^{s}}\,\Big(\mathcal{I}_1+\mathcal{I}_2\Big).
\end{split}
\end{equation}
In turn, the term $\mathcal{I}_1$ can be bounded by $\mathcal{I}_2$. Indeed, 
	by observing that the integrand of $\mathcal{I}_1$ is constant and using the normalization \eqref{normalizza}, we get
	\[
	\begin{split}
		\mathcal{I}_1=\sqrt{|B_R|}\,|1-\mathrm{av}(\widetilde u;B_R)|
		&=\sqrt{\frac{|B_R|}{|Q_r|}}\,\Big|\|u\|_{L^2(Q_r)}-\|\mathrm{av}(\widetilde u;B_R)\|_{L^2(Q_r)}\Big|\\
		&\le \sqrt{\frac{|B_R|}{|Q_r|}}\,\left\|u-\mathrm{av}(\widetilde u;B_R)\right\|_{L^2(Q_r)}\le \sqrt{\frac{|B_R|}{|Q_r|}}\,\mathcal{I}_2.
	\end{split}
	\]
	As for the integral $\mathcal{I}_2$, by Lemma \ref{lm:PW} we directly get
	\[
	\mathcal{I}_2\le \sqrt{\mathcal{M}\,(1-s)}\,R^s\,[\widetilde u]_{W^{s,2}(B_R)}.
	\]
Then the last term can be estimated by \eqref{stimainiziale}, again. 
By inserting these estimates in \eqref{insiste} we eventually conclude the proof.
\end{proof}

\subsection{A geometric lower bound in the plane}

In dimension $N=2$ and for $s>1/2$, by exploiting the fact that points have positive relative fractional capacity (see Lemma \ref{cor: capacità punto risp segmento}), it is possible to give a geometric lower bound for the term
\[
\widetilde{\mathrm{cap}}_s(\Sigma;B_R(x_0)),
\]
appearing in \eqref{eq: stima capacità}. We will follow the idea of \cite[Chapter 3, Section 1.2, Proposition 1]{Ma},
which is quite close to that used by Taylor, even if the latter worked with a different notion of {\it capacity} coming from Potential Theory. The proof will also crucially exploits the result on ``directional'' fractional derivatives (Proposition \ref{prop:direzioni} and Corollary \ref{coro:francesca}). We still use the symbol $\Pi_\omega$ defined in \eqref{projection}.

 \begin{prop}
\label{prop: capacità 2d VS capacità 1d}
	Let $N=2$, $1/2<s<1$ and let $\Sigma\Subset B_r(x_0)$ be a compact set. For every direction $\omega\in\mathbb{S}^1$, it holds that 
	\[
	\widetilde{\mathrm{cap}}_s(\Sigma;B_r(x_0))
	\ge 
	\frac{\mathfrak{m}_s}{\mathcal{A}}\,\Big(r\,\mathrm{dist}(\Sigma,\partial B_r(x_0)\Big)^\frac{1-2\,s}{2}\,\mathcal{H}^1(\Pi_\omega(\Sigma)).
	\]
	Here $\mathcal{A}$ is the same constant as in Proposition \ref{prop:direzioni} and $\mathfrak{m}_s$ is the same constant as in Theorem \ref{teo:MC}.
\end{prop}

\begin{proof}
We observe that we can assume $\mathcal{H}^1(\Pi_\omega(\Sigma))>0$, otherwise there is nothing to prove. We may suppose as always that $x_0=0$, without loss of generality.
\par
We start by noticing that every $x\in \mathbb{R}^2$ can be written as
\[
x=x'+t\,\omega,\qquad \mbox{ with } x'\in\Pi_\omega(\mathbb{R}^2)\ \mbox{and}\ t\in \mathbb{R}.
\] 
We also set 
\[
R_\omega(x')=\sup\Big\{\varrho\in\mathbb{R}\, :\, x'+\varrho\,\omega\in B_r\Big\}\qquad \mbox{ and }\qquad r_\omega(x')=\inf\Big\{\varrho\in\mathbb{R}\, :\, x'+\varrho\,\omega\in B_r\Big\}.
\]
We take $u\in C^\infty_0(B_r)$ such that $u\ge 1_\Sigma$. By using Corollary \ref{coro:francesca} and Fubini's Theorem, we can infer
	\begin{equation}\label{eq: seminorma palle con sigma_omega}
	\begin{split}
		[u]^2_{W^{s,2}(\mathbb{R}^2)}&\ge \frac{1}{\mathcal{A}}\,
		\int_{\mathbb{R}^2}\left(\int_{\mathbb{R}} \frac{|u(x)-u(x+\varrho\,\omega)|^2}{|\varrho|^{1+2\,s}}\,d\varrho\right)\,dx\\
		&=\frac{1}{\mathcal{A}}\,
		\int_{\Pi_\omega(\mathbb{R}^2)}\left(\iint_{\mathbb{R}\times \mathbb{R}} \frac{|u(x'+t\,\omega)-u(x'+t\,\omega+\varrho\,\omega)|^2}{|\varrho|^{1+2\,s}}\,dt\,d\varrho\right)\,dx'\\
			&\ge \frac{1}{\mathcal{A}}\,\int_{\Pi_\omega(\Sigma)}[u(x'+\cdot\,\omega)]^2_{W^{s,2}(\mathbb{R})}\,dx'.
	\end{split}
	\end{equation}
	Recalling that $u\ge1$ on $\Sigma$, it follows that for every $x'\in \Pi_\omega(\Sigma)$ there exists $t_0$ such that $u(x'+t_0\,\omega)\ge1$. Hence, by using the trial function
\[
\psi_{x'}=u(x'+\cdot\,\omega)\in C^\infty_0((r_\omega(x'),R_\omega(x'))),
\]	
we have
	\[
	[u(x'+\cdot\,\omega)]^2_{W^{s,2}(\mathbb{R})}=[\psi_{x'}]^2_{W^{s,2}(\mathbb{R})}\ge \widetilde{\mathrm{cap}}_s\big(\{t_0\};(r_\omega(x'),R_\omega(x')\big),\qquad \mbox{ for } x'\in \Pi_\omega(\Sigma),
	\]
	by the very definition of capacity.
In turn, by applying Lemma \ref{cor: capacità punto risp segmento} in the right-hand side above, we get 
\[
[u(x'+\cdot\,\omega)]^2_{W^{s,2}(\mathbb{R})}\ge \mathfrak{m}_s\,\Big(R_\omega(x')-r_\omega(x')\Big)^{1-2\,s}.
\]
In order to get a lower bound for the last term, we set $\ell=\mathrm{dist}(\Sigma,\partial B_r)>0$. Then in particular we have
\[
R_\omega(x')-r_\omega(x')\ge \sqrt{r^2-(r-\ell)^2}\ge \sqrt{r\,\ell},\qquad \mbox{ for every } x'\in \Pi_\omega(\Sigma).
\]
This entails that 
\[
[u(x'+\cdot\,\omega)]^2_{W^{s,2}(\mathbb{R})}\ge \mathfrak{m}_s \,(r\,\ell)^\frac{1-2\,s}{2},\qquad \mbox{ for every } x'\in \Pi_\omega(\Sigma).
\]
By spending this information in \eqref{eq: seminorma palle con sigma_omega}, we can obtain
\[
[u]^2_{W^{s,2}(B_r)}
\ge \frac{\mathfrak{m}_s}{\mathcal{A}}\,(r\,\ell)^\frac{1-2\,s}{2}\,\mathcal{H}^1(\Pi_\omega(\Sigma)).
\]
The thesis follows by taking the infimum over the admissible trial functions $u$.
\end{proof}

\section{Proof of Theorem \ref{teo:main}}
\label{sec:5}

	Without loss of generality, we can assume $r_\Omega =1$. As in the proof of Lemma \ref{lm:rettangolo distante}, we consider the natural number $\delta=\lfloor\sqrt{k}\rfloor+1$ and take the family of squares $\{\mathcal{Q}_{i,j}\}_{(i,j)\in\mathbb{Z}^2}\subseteq\mathbb{R}^2$ given by
	\[
	\mathcal{Q}_{i,j}:=Q_{5\,\delta}(10\,\delta\,i,10\,\delta\,j),\qquad \mbox{ for }(i,j)\in\mathbb{Z}^2.
	\]
	We observe that they form a tiling of the whole plane, more precisely they are pairwise disjoint and the union of their closures covers the whole $\mathbb{R}^2$. We also introduce the set of indexes 
\[
\mathbb{Z}^2_\Omega=\Big\{(i,j)\in\mathbb{Z}^2\, :\, \mathcal{Q}_{i,j}\cap \Omega\neq\emptyset \Big\},	
\]
and for every $(i,j)\in\mathbb{Z}^2_\Omega$, we indicate by $\Sigma_{i,j}\subset \overline{\mathcal{Q}_{i,j}}\setminus\Omega$ the compact set provided by Lemma \ref{lm:rettangolo distante}. By using the tiling properties of these squares, for a function $u\in C^\infty_0(\Omega)$ we have
	\[
	\begin{split}
		[u]^2_{W^{s,2}(\mathbb{R}^2)}&=\sum_{(i,j)\in\mathbb{Z}^2}\iint_{\mathcal{Q}_{i,j}\times \mathbb{R}^2}\frac{|u(x)-u(y)|^2}{|x-y|^{2+2\,s}}\,dx\,dy
		\\
		&\ge \sum_{(i,j)\in\mathbb{Z}^2}\iint_{\mathcal{Q}_{i,j}\times\mathcal{Q}_{i,j}}\frac{|u(x)-u(y)|^2}{|x-y|^{2+2\,s}}\,dx\,dy= \sum_{(i,j)\in\mathbb{Z}^2_\Omega} [u]^2_{W^{s,2}(\mathcal{Q}_{i,j})}.
	\end{split}
	\]	
For every $(i,j)\in\mathbb{Z}^2_\Omega$, we can use the fractional Poincar\'e inequality of Proposition \ref{lm:poincare_cap}, with the choices 
\[
r=5\,\delta\qquad \mbox{ and }\qquad R=2\,r=10\,\delta.
\] 
By setting for brevity $\mathcal{B}_{i,j}:=B_{10\,\delta}(10\,\delta\,i,10\,\delta j)$, this leads to
\[
[u]^2_{W^{s,2}(\mathcal{Q}_{i,j})}\ge \left[\frac{1}{50\,\delta^2}\,\phi\left(2,2\right)\right]\,\widetilde{\mathrm{cap}}_s\Big(\Sigma_{i,j};\mathcal{B}_{i,j}\Big)\,\|u\|^2_{L^2(\mathcal{Q}_{ij})},\qquad \mbox{ for every } (i,j)\in\mathbb{Z}^2_\Omega,
\]
where we also used that $s>1/2$. We now have to estimate from below the relative fractional capacity of each compact set $\Sigma_{i,j}$. By combining
Lemma \ref{lm:rettangolo distante} and Proposition \ref{prop: capacità 2d VS capacità 1d}, we have
\[
\begin{split}
\widetilde{\mathrm{cap}}_s(\Sigma_{i,j};\mathcal{B}_{i,j})
	&\ge \left(50\,(2-\sqrt{2})\right)^\frac{1-2\,s}{2}\,\frac{\mathfrak{m}_s}{\mathcal{A}}\,\delta^{1-2\,s}\, \max\Big\{\mathcal{H}^1(\Pi_{\mathbf{e}_1}(\Sigma_{i,j})),\, \mathcal{H}^1(\Pi_{\mathbf{e}_2}(\Sigma_{i,j}))\Big\}\\
	&	\ge \left(50\,(2-\sqrt{2})\right)^\frac{1-2\,s}{2}\,\frac{\mathfrak{m}_s}{4\,\mathcal{A}}\,\delta^{1-2\,s}\,\sqrt{k}.
\end{split}
\]
By collecting the estimates above, we obtain
\begin{equation}
\label{quasifinito}
	\begin{split}
	[u]^2_{W^{s,2}(\mathbb{R}^2)}&\ge \left(50\,(2-\sqrt{2})\right)^\frac{1-2\,s}{2}\,\frac{\mathfrak{m}_s\,\phi(2,2)}{200\,\mathcal{A}}\,\sqrt{k}\,\delta^{-1-2\,s}\,\sum_{(i,j)\in\mathbb{Z}^2_\Omega}\|u\|^2_{L^2(\mathcal{Q}_{ij})}\\
	&=\left(50\,(2-\sqrt{2})\right)^\frac{1-2\,s}{2}\,\frac{\mathfrak{m}_s\,\phi(2,2)}{200\,\mathcal{A}}\,\sqrt{k}\,\delta^{-1-2\,s}\,\|u\|^2_{L^2(\Omega)},
	\end{split}
\end{equation}	
where the last identity follows by the tiling property of the family $\{\mathcal{Q}_{ij}\}_{i,j}$. By recalling the definition of $\delta$ and using \eqref{mantissa}, we get
	\[
	\sqrt{k}\,\delta^{-1-2\,s}\ge \sqrt{k}\,\left(\sqrt{k}+1\right)^{-1-2\,s}\ge\frac{1}{2^{1+2\,s}}\,\frac{1}{k^s}.
	\]
By the arbitrariness of $u\in C^\infty_0(\Omega)$, from \eqref{quasifinito} we get the claimed lower bound on $\lambda_1^s(\Omega)$, with
\[
\vartheta_s=\frac{\left(50\,(2-\sqrt{2})\right)^\frac{1-2\,s}{2}}{2^{1+2\,s}}\,\frac{\mathfrak{m}_s\,\phi(2,2)}{200\,\mathcal{A}}.
\]
Finally, the claimed asymptotic behaviour of $\vartheta_s$ simply follows from its definition and the properties of $\mathfrak{m}_s$, encoded in Theorem \ref{teo:MC}.

\section{Proof of Theorem \ref{teo:optimal}}
\label{sec:6}

\subsection{Proof of point (1)}
This is a straightforward consequence of the {\it Bourgain-Brezis-Mironescu formula}. Indeed, for every $\Omega\subseteq\mathbb{R}^2$ open set, let $u\in C^\infty_0(\Omega)\setminus\{0\}$. Then by \cite[Corollary 3.20]{EE} we have 
\[
\lim_{s\nearrow 1} (1-s)\,[u]_{W^{s,2}(\mathbb{R}^2)}^2=\frac{1}{2}\, \int_\Omega|\nabla u|^2\,dx.
\]
This implies that 
\[
\limsup_{s\nearrow 1} (1-s)\,\lambda_1^s(\Omega)\le \lim_{s\nearrow 1} \frac{(1-s)\,[u]_{W^{s,2}(\mathbb{R}^N)}^2}{\|u\|_{L^2(\Omega)}^2}=\frac{1}{2}\, \frac{\displaystyle\int_\Omega|\nabla u|^2\,dx}{\|u\|_{L^2(\Omega)}^2}.
\]
By taking the infimum over $C^\infty_0(\Omega)\setminus\{0\}$, we get
\[
\limsup_{s\nearrow 1} (1-s)\,\lambda_1^s(\Omega)\le \frac{1}{2}\,\lambda_1(\Omega),
\]
as claimed. Thus, by multiplying both sides of \eqref{fracOTC} by the factor $(1-s)$, using the previous property and the asymptotic behaviour of $\vartheta_s$, we get back the classical Croke-Osserman-Taylor estimate, in the limit as $s$ goes to $1$.

\subsection{Proof of point (2)}
We need at first the following 
\begin{lemma}
\label{lm:capacitapunti}
Let $0<s<1$ and let $\Omega\subseteq\mathbb{R}^2$ be an open set. Then for every $\{x_0,\dots,x_m\}\subset \Omega$, we have 
\[
\lambda_1^s\big(\Omega\setminus\{x_0,\dots,x_m\}\big)=\lambda_1^s(\Omega).
\]
\end{lemma}
\begin{proof}
We may suppose that the points $\{x_0,\dots,x_m\}$ are distinct. We first observe that 
\[
\lambda_1^s\big(\Omega\setminus\{x_0,\dots,x_m\}\big)\ge\lambda_1^s(\Omega),
\]
since $\Omega\setminus\{x_0,\dots,x_m\}\subset \Omega$. In order to prove the converse implication, we set
\[
\varepsilon_0=\frac{1}{2}\,\min_{i,j\in\{0,m\}} \Big\{|x_i-x_j|\, :\, i\not=j\Big\}.
\] 
Then we take a cut-off function $\eta\in C^\infty_0(B_1)$ such that 
\[
\eta\equiv 1 \mbox{ in } B_\frac{1}{2},\qquad 0\le \eta\le 1,\qquad |\nabla \eta|\le C,
\] 
and  define for every $0<\varepsilon<\varepsilon_0$
\[
\Psi_\varepsilon(x)=\sum_{i=0}^m \eta\left(\frac{x-x_i}{\varepsilon}\right).
\]
We now take $u\in C^\infty_0(\Omega)\setminus\{0\}$ and observe that $u\,(1-\Psi_\varepsilon)$ is a feasible trial function for the variational problem which defines $\lambda_1^s(\Omega\setminus\{x_0,\dots,x_m\})$. Thus, by using Minkowski's inequality, we get for every $0<\varepsilon<\varepsilon_0$
	\begin{equation}
	\label{eq: radice lambda}
	\begin{split}
		\sqrt{\lambda_1^s(\Omega\setminus\{x_0,\dots,x_m\})}&\le\frac{\big[u\,(1-\Psi_{\varepsilon})\big]_{W^{s,2}(\mathbb{R}^2)}}{\big\|u\,(1-\Psi_{\varepsilon})\big\|_{L^2(\Omega)}}\\
				&\le \frac{[u]_{W^{s,2}(\mathbb{R}^2)}\,\|1-\Psi_{\varepsilon}\|_{L^\infty(\mathbb{R}^2)}+\|u\|_{L^\infty(\mathbb{R}^2)}\,[\Psi_{\varepsilon}]_{W^{s,2}(\mathbb{R}^2)} }{\|u\,(1-\Psi_{\varepsilon})\|_{L^2(\Omega)}}\\
				&=\frac{[u]_{W^{s,2}(\mathbb{R}^2)}+\|u\|_{L^\infty(\mathbb{R}^2)}\,[\Psi_{\varepsilon}]_{W^{s,2}(\mathbb{R}^2)} }{\|u\,(1-\Psi_{\varepsilon})\|_{L^2(\Omega)}}.
				\end{split}
				\end{equation}
			By applying the Dominated Convergence Theorem, we easily get that
			\[
			\lim_{\varepsilon\to 0}\|u\,(1-\Psi_{\varepsilon})\|_{L^2(\Omega)}=\|u\|_{L^2(\Omega)}.
			\]	
			As for the second term in the numerator, we observe that by Minkowski's inequality again, we have
			\[
			[\Psi_{\varepsilon}]_{W^{s,2}(\mathbb{R}^2)}=\left[\sum_{i=0}^m \eta\left(\frac{\cdot-x_i}{\varepsilon}\right)\right]_{W^{s,2}(\mathbb{R}^2)}\le (m+1)\,\varepsilon^{1-s}\,[\eta]_{W^{s,2}(\mathbb{R}^2)}.
			\] 
			We also used the scaling properties of the fractional seminorm. This in turn implies that 
			\[
			\lim_{\varepsilon\to 0}[\Psi_{\varepsilon}]_{W^{s,2}(\mathbb{R}^2)}=0.
			\]
			Thus, by taking the limit as $\varepsilon$ goes to $0$ in \eqref{eq: radice lambda}, we end up with
			\[
			\sqrt{\lambda_1^s(\Omega\setminus\{x_0,\dots,x_m\})}\le \frac{[u]_{W^{s,2}(\mathbb{R}^2)}}{\|u\|_{L^2(\Omega)}}.
			\]
			By arbitrariness of $u$, we get the desired conclusion.
\end{proof}
\begin{oss}
The previous result is a particular case of the following general fact: removing sets with zero fractional capacity does not alter the relevant fractional Sobolev space. Consequently, fractional Poincar\'e constants are insensitive to removal of these sets. We refer for example to \cite[Proposition 2.6 \& Corollary 2.7]{AFN} for this general result.
\end{oss}
The sequence $\{\Omega_k\}_{k\in\mathbb{N}\setminus\{0,1\}}$ is then constructed as follows: for every $k\in\mathbb{N}\setminus\{0,1\}$, we set
\[
n_k=\lfloor \sqrt{k-1}\rfloor\qquad \mbox{ and }\qquad m_k=(k-1)-n_k^2.
\]
Then, we take the set 
\[
\mathrm{Shell}_k=\left(\Big[0,n_k\Big]\times\Big[0,n_k\Big]\right)\setminus \bigcup_{i,j=0}^{n_k-1}\left\{\left(i+\frac{1}{2},j+\frac{1}{2}\right)\right\}, \mbox{ for }k\ge 2,
\]
which consists of a square with $n_k^2$ equally spaced points removed. More precisely, we remove the centers of the squares
\[
[i,i+1]\times[j,j+1],\qquad \mbox{ for } i,j=0,\dots,n_k-1.
\] 
We also introduce the set
\[
\mathrm{Slug}_k=\Big([0,m_k]\times[-1,0]\Big)\setminus \bigcup_{i=0}^{m_k-1}\left\{\left(i+\frac{1}{2},-\frac{1}{2}\right)\right\},
\]
which consists of an horizontal strip of width $1$ and length $m_k$, from which we removed the centers of the squares
\[
[i,i+1]\times[-1,0],\qquad \mbox{ for } i=0,\dots,m_k-1.
\]
	\begin{figure}
\includegraphics[scale=.3]{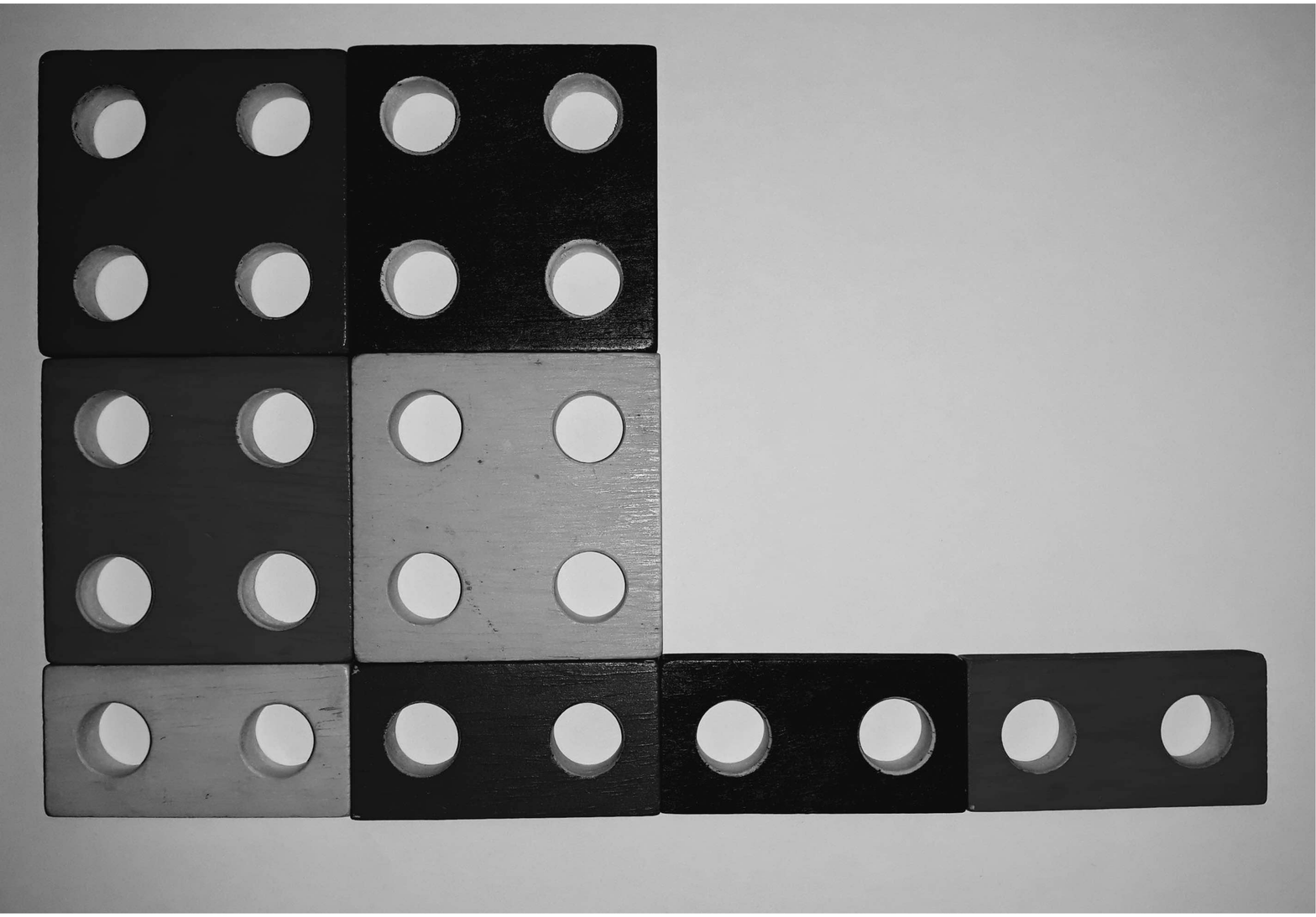}
\caption{The set $\Omega_k$ of Theorem \ref{teo:optimal} point (2), for $k=25$.}
\label{fig:lumaca}
\end{figure}
Finally, we define the open bounded set
\[
\Omega_k=\mathrm{int}\big(\mathrm{Shell}_k\cup \mathrm{Slug}_k\big),\qquad \mbox{ for every } k\ge 2,
\]
i.\,e. the interior points of the union of $\mathrm{Shell}_k$ and $\mathrm{Slug}_k$ (see Figure \ref{fig:lumaca}). By construction, we have that $\Omega_k$ is multiply connected of order $k$. Moreover, we have 
\[
r_{\Omega_k}\le \frac{\sqrt{2}}{2},\qquad \mbox{ for every } k\ge 2,
\]
and
\[
\Omega_k\supseteq \mathrm{int}\big(\mathrm{Shell}_k\big)=\left(\Big(0,n_k\Big)\times\Big(0,n_k\Big)\right)\setminus \bigcup_{i,j=0}^{n_k-1}\left\{\left(i+\frac{1}{2},j+\frac{1}{2}\right)\right\}.
\]
By using the monotonicity of $\lambda_1^s$ with respect to set inclusion and Lemma \ref{lm:capacitapunti} for $\mathrm{int}\big(\mathrm{Shell}_k\big)$, we can then infer
\[
\lambda_1^s(\Omega_k)\le \lambda_1^s\left(\Big(0,n_k\Big)\times\Big(0,n_k\Big)\right)=n_k^{-2\,s}\,\lambda_1^s\left(\Big(0,1\Big)\times\Big(0,1\Big)\right).
\] 	
By recalling the definition of $n_k$, this finally gives the desired result.

\subsection{Proof of point (3)} We divide the proof in various steps, for ease of presentation.
\vskip.2cm\noindent
{\it Step 1: construction of the set}.
We define
\[
{\Sigma}=\bigcup\limits_{i\in\mathbb{Z}}{\Sigma}^{(i)},\qquad \mbox{ where }\, \Sigma^{(i)}:=\left\{(x_1,i)\in\mathbb{R}^2\,:\,|x_1|\ge1\right\},
\] 
and then consider the {\it infinite complement comb} 
\[
\Theta:=\mathbb{R}^2\setminus \Sigma,
\]
as in \cite[Section 5]{BB}. The set $\Theta_k$ of the statement is then constructed by simply removing $k-1$ distinct points from $\Theta$, i.e. we set
\[
\Theta_k=\Theta\setminus\{(0,i)\, :\, i=1,\dots,k-1\},
\]
see Figure \ref{fig:pettinone}.
\begin{figure}
\includegraphics[scale=.3]{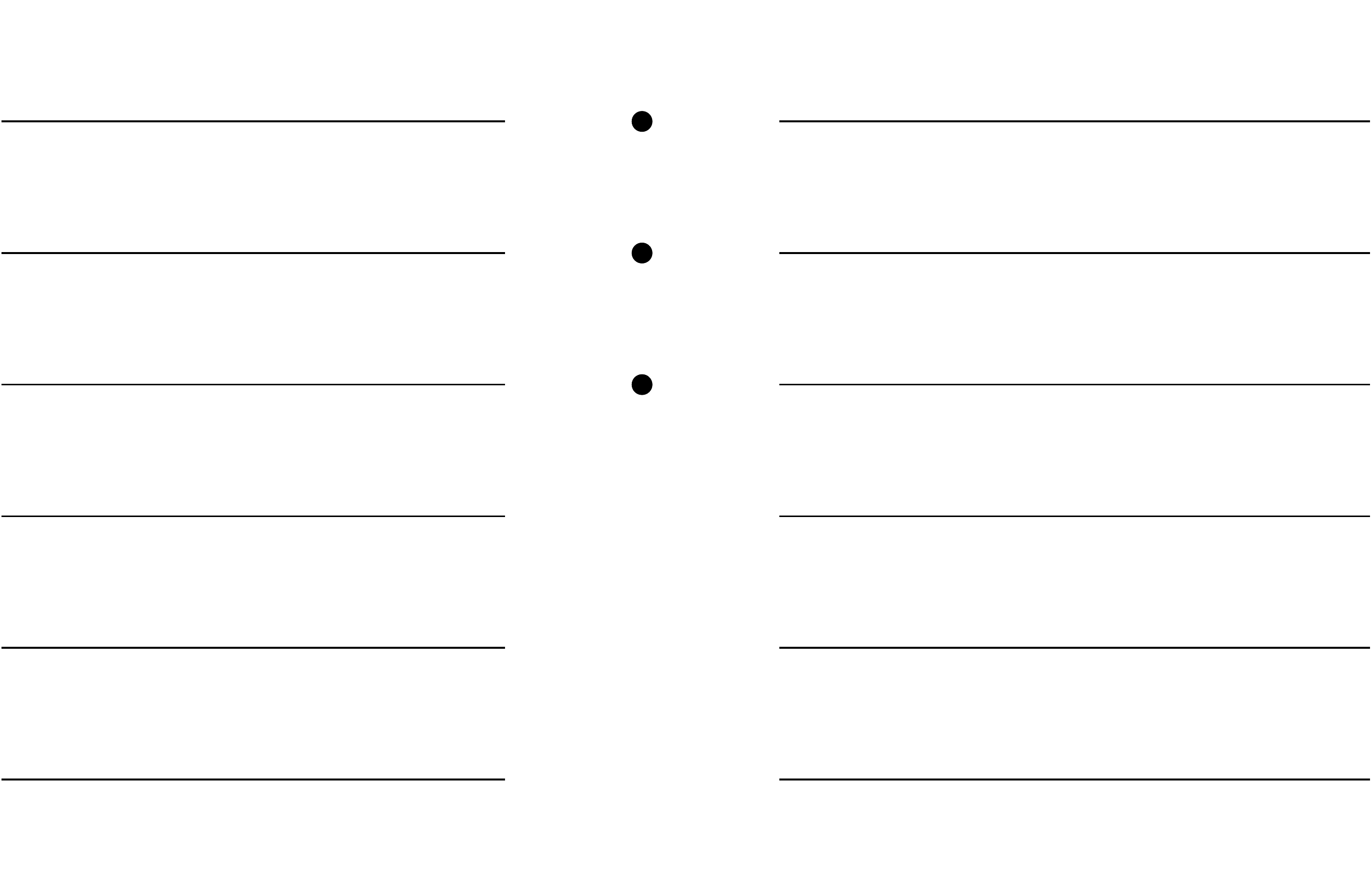}
\caption{The set $\Theta_k$ for $k=4$ of Theorem \ref{teo:optimal} point (3): it is has been obtained by removing the black dots from $\Theta$.}
\label{fig:pettinone}
\end{figure}
By construction, we have that $\Theta_k$ is multiply connected of order $k$, with finite inradius. Thus, by Theorem \ref{teo:main} we have $\lambda_1^s(\Theta_k)>0$, for every $s>1/2$. 
We claim that 
\begin{equation}
\label{step1}
\limsup_{s\searrow \frac{1}{2}} \frac{\lambda_1^s(\Theta_k)}{2\,s-1}<+\infty.
\end{equation}
{\it Step 2: one-dimensional reduction}. Here we need the following result.
\begin{lemma}
Let $0<s<1$ and let $A\subseteq\mathbb{R}$ be an open set. Then we have
\begin{equation}
\label{alphas}
\lambda_1^s(A\times\mathbb{R})\le \alpha_s\,\lambda_1^s(A),\qquad \mbox{ where }\alpha_s=\int_\mathbb{R} \frac{d t}{(1+t^2)^\frac{2+2\,s}{2}}.
\end{equation}
\end{lemma}
\begin{proof}
We proceed as in the proof of \cite[Lemma 2.4]{FS}.	For every $x\in\mathbb{R}^2$, we will use the notation $x=(x_1,x_2)$. We take $u\in C^\infty_0(A)\setminus\{0\}$ and $\varphi\in C^\infty_0(\mathbb{R})\setminus\{0\}$. We first observe that by Fubini's Theorem
	\[
	\|u\,\varphi\|_{L^2(A\times\mathbb{R})}=\|u\|_{L^2(A)}\,\|\varphi\|_{L^2(\mathbb{R})}.
	\]
We then estimate the fractional seminorm of $u\,\varphi$. By Minkowksi's inequality, we have
\[
\begin{split}
[u\,\varphi]_{W^{s,2}(\mathbb{R}^2)}&\le \left(\iint_{\mathbb{R}^{2}\times\mathbb{R}^{2}} |u(x_1)|^2\,\dfrac{|\varphi(y_1)-\varphi(y_2)|^2}{|x-y|^{2+2\,s}}\,dx\,dy\right)^{1/2}\\
&+\left(\iint_{\mathbb{R}^{2}\times\mathbb{R}^{2}} |\varphi(y_2)|^2\,\dfrac{|u(x_1)-u(x_2)|^2}{|x-y|^{2+2\,s}}\,dx\,dy\right)^{1/2}.
\end{split}
\]	
By using Fubini's Theorem, we have
\[
\begin{split}
\iint_{\mathbb{R}^{2}\times\mathbb{R}^{2}}& |u(x_1)|^2\,\dfrac{|\varphi(y_1)-\varphi(y_2)|^2}{|x-y|^{2+2\,s}}\,dx\,dy\\
&=\int_\mathbb{R} |u(x_1)|^2\,\left(\iint_{\mathbb{R}\times\mathbb{R}} |\varphi(y_1)-\varphi(y_2)|^2\left(\int_\mathbb{R} \frac{dx_2}{\left((x_1-x_2)^2+(y_1-y_2)^2\right)^\frac{2+2\,s}{2}}\right)\,d y_1\,dy_2\right)\,dx_1.
\end{split}
\]
By using a changing of variable, we get
\[
\int_\mathbb{R} \frac{dx_2}{\left((x_1-x_2)^2+(y_1-y_2)^2\right)^\frac{2+2\,s}{2}}=\frac{\alpha_s}{|y_1-y_2|^{1+2\,s}}.
\]
Thus we obtain
\[
\iint_{\mathbb{R}^{2}\times\mathbb{R}^{2}} |u(x_1)|^2\,\dfrac{|\varphi(y_1)-\varphi(y_2)|^2}{|x-y|^{2+2\,s}}\,dx\,dy=\alpha_s\,\|u\|^2_{L^2(A)}\,[\varphi]^2_{W^{s,2}(\mathbb{R})}.
\]
With a similar computation, we also get
\[
\iint_{\mathbb{R}^{2}\times\mathbb{R}^{2}} |\varphi(y_2)|^2\,\dfrac{|u(x_1)-u(x_2)|^2}{|x-y|^{2+2\,s}}\,dx\,dy=\alpha_s\,\|\varphi\|^2_{L^2(\mathbb{R})}\,[u]^2_{W^{s,2}(\mathbb{R})}.
\]
Thus, from the variational definition of $\lambda_1^s(A\times\mathbb{R})$, we get
\[
\begin{split}
\sqrt{\lambda_1^s(A\times\mathbb{R})}\le \frac{[u\,\varphi]_{W^{s,2}(\mathbb{R}^2)}}{\|u\,\varphi\|_{L^2(\omega\times\mathbb{R})}}&\le \sqrt{\alpha_s}\, \frac{\|u\|_{L^2(A)}\,[\varphi]_{W^{s,2}(\mathbb{R})}+\|\varphi\|_{L^2(\mathbb{R})}\,[u]_{W^{s,2}(\mathbb{R})}}{\|u\|_{L^2(A)}\,\|\varphi\|_{L^2(\mathbb{R})}}\\
&=\sqrt{\alpha_s}\,\left(\frac{[\varphi]_{W^{s,2}(\mathbb{R})}}{\|\varphi\|_{L^2(\mathbb{R})}}+\frac{[u]_{W^{s,2}(\mathbb{R})}}{\|u\|_{L^2(A)}}\right).
\end{split}
\]
By taking the infimum over $u$ and $\varphi$, recalling that $\lambda^s_1(\mathbb{R})=0$, we get the desired conclusion
\end{proof}
In particular, from the previous result with $A=\mathbb{R}\setminus\mathbb{Z}$, we get that 
\[
\lambda_1^s(\Theta_k)\le \lambda_1^s(\mathbb{R}\times(\mathbb{R}\setminus \mathbb{Z}))\le \alpha_s\,\lambda_1^s(\mathbb{R}\setminus\mathbb{Z}).
\]
In the first inequality we used that 
\[
\mathbb{R}\times(\mathbb{R}\setminus \mathbb{Z})\subset \Theta_k.
\]
From its definition \eqref{alphas}, it is easy to see that $\alpha_s$ various continuously with respect to $s\in[0,1]$. Thus, in order to prove \eqref{step1}, it will be sufficient to establish that 
\begin{equation}
\label{step2}
\limsup_{s\searrow \frac{1}{2}} \frac{\lambda_1^s(\mathbb{R}\setminus\mathbb{Z})}{2\,s-1}<+\infty.
\end{equation}
\noindent
{\it Step 3: choice of the trial functions}.
In order to prove \eqref{step2}, we will need to carefully construct a suitable family of $s-$depending trial functions, which provides an upper bound on $\lambda_1^s(\mathbb{R}\setminus\mathbb{Z})$ with the correct asymptotic behaviour. For every 
\[
n\in\mathbb{N}\setminus\{0\},\qquad s>1/2\qquad \mbox{ and }\qquad 0<\varepsilon<\frac{1}{10},
\] 
we consider the trial function $u_{n}\,\varphi_{n,s,\varepsilon}$,
where:
\begin{itemize}
\item $u_{n}\in C^\infty_0((-n,n))$ has the form
\[
u_{n}(x)=u\left(\frac{x}{n}\right),
\]
for some $u\in C^\infty_0((-1,1))$ such that $\|u\|_{L^2((-1,1))}=1$;
\vskip.2cm
\item the {\it multiple funnel--type} cut-off function $\varphi_{n,s,\varepsilon}\in W^{s,2}_{\rm loc}(\mathbb{R})\cap L^\infty(\mathbb{R})$ has the form
\[
\varphi_{n,s,\varepsilon}=1-\sum_{j=-n}^n \zeta_s\left(\frac{x-j}{\varepsilon}\right), 
\]
where $\zeta_s$ is the function given by
\[
\zeta_s(x)=\Big(1-|x|^{2\,s-1}\Big)_+.
\]
\end{itemize}
Thanks to \cite[Lemma 2.7]{BBZ}, we see that
\[
u_{n}\,\varphi_{n,s,\varepsilon}\in \widetilde{W}^{s,2}_0((-n,n))\subset \widetilde{W}^{s,2}_0(\mathbb{R}\setminus\mathbb{Z}).
\]
Thus it is a feasible trial function. By using again Minkowski's inequality, this yields
\[
\sqrt{\lambda_1^s(\mathbb{R}\setminus\mathbb{Z})}\le \frac{[u_n]_{W^{s,2}(\mathbb{R})}+\|u_n\|_{L^\infty((-n,n))}\,[\varphi_{n,\varepsilon,s}]_{W^{s,2}(\mathbb{R})}}{\|u_n\,\varphi_{n,\varepsilon,s}\|_{L^2((-n,n))}}.
\]
\noindent
{\it Step 4: estimate of the quotient.} Let us start by handling the terms at the numerator. We consider at first the terms containing $u_n$, which are simpler. 
By recalling the definition of $u_n$, we have
\[
[u_n]_{W^{s,2}(\mathbb{R})}=n^{\frac{1}{2}-s}\,[u]_{W^{s,2}(\mathbb{R})}.
\]
The last term can be estimated by using the interpolation inequality \cite[Corollary 2.2]{BPS}, which gives
\[
[u]_{W^{s,2}(\mathbb{R})}\le \sqrt{\frac{C}{s\,(1-s)}}\,\|u\|_{L^2((-1,1))}^{1-s}\,\|u'\|_{L^2((-1,1))}^{s},
\]
for some $C>0$ independent of $s$. This guarantees that we have 
\begin{equation}
\label{frankfurt1}
[u_n]_{W^{s,2}(\mathbb{R})}\le n^{\frac{1}{2}-s}\,\sqrt{\frac{C}{s\,(1-s)}}\,\|u'\|_{L^2((-1,1))}^{s}.
\end{equation}
The term with the $L^\infty$ norm is easy to handle, since we simply have
\begin{equation}
\label{frankfurt2}
\|u_n\|_{L^\infty((-n,n))}=\|u\|_{L^\infty((-1,1))}.
\end{equation}
The term containing the cut-off is the most delicate one. In order to estimate it, we observe that 
\[
\sum_{j=-n}^n \zeta_s\left(\frac{x-j}{\varepsilon}\right)=\max_{j=-n,\dots,n} \zeta_s\left(\frac{x-j}{\varepsilon}\right),
\] 
thanks to the fact that all the functions involved in the sum have disjoint support. We can then use the sub-modularity of the Sobolev-Slobodecki\u{\i} seminorm (see \cite[Theorem 3.2 \& Remark 3.3]{GM}) and obtain
\[
\begin{split}
[\varphi_{n,\varepsilon,s}]_{W^{s,2}(\mathbb{R})}&=\left[\sum_{j=-n}^n \zeta_s\left(\frac{\cdot-j}{\varepsilon}\right)\right]_{W^{s,2}(\mathbb{R})}\\&=\left[\max_{j=-n,\dots,n} \zeta_s\left(\frac{\cdot-j}{\varepsilon}\right)\right]_{W^{s,2}(\mathbb{R})}\\
&\le \left(\sum_{j=-n}^n\left[\zeta_s\left(\frac{\cdot-j}{\varepsilon}\right)\right]^2_{W^{s,2}(\mathbb{R})} \right)^\frac{1}{2}=\sqrt{2\,n+1}\,\varepsilon^{\frac{1}{2}-s}\,[\zeta_s]_{W^{s,2}(\mathbb{R})}.
\end{split}
\]
In order to conclude, the key fact is a very precise asymptotic estimate of the last term, as $s$ goes to $1/2$.  This is contained in Lemma \ref{lm:stimonaMFO} below, which permits to infer 
\begin{equation}
\label{frankfurt3}
[\varphi_{n,\varepsilon,s}]_{W^{s,2}(\mathbb{R})}\le C\,\sqrt{2\,n+1}\,\varepsilon^{\frac{1}{2}-s}\,\sqrt{2\,s-1}, \qquad \mbox{ for } \frac{1}{2}<s<\frac{3}{4},
\end{equation}
with $C>0$ not depending on $s$. 
\par
We now pass to examine the denominator. In this case, we have
\begin{equation}
\label{frankfurt4}
\begin{split}
\|u_n\,\varphi_{n,\varepsilon,s}\|_{L^2((-n,n))}&=n^\frac{1}{2}\,\left(\int_{-1}^1 |u(y)|^2\,\left(1-\sum_{j=-n}^n \zeta_s\left(\frac{n\,y-j}{\varepsilon}\right)\right)^2\,dy\right)^\frac{1}{2}\ge n^\frac{1}{2}\,\|u\|_{L^2(A_\varepsilon)},
\end{split}
\end{equation}
where
\[
A_\varepsilon=(-1,1)\setminus \bigcup\limits_{j=-n}^n \left(\frac{j-\varepsilon}{n},\frac{j+\varepsilon}{n}\right).
\]
\noindent
{\it Step 5: conclusion.}
By collecting the estimates \eqref{frankfurt1}, \eqref{frankfurt2}, \eqref{frankfurt3} and \eqref{frankfurt4}, we obtain
\[
\begin{split}
\sqrt{\frac{\lambda_1^s(\mathbb{R}\setminus\mathbb{Z})}{2\,s-1}}&\le \frac{\displaystyle n^{\frac{1}{2}-s}\,\sqrt{\frac{C}{s\,(1-s)}}\,\|u'\|_{L^2((-1,1))}^{s}+C\,\|u\|_{L^\infty((-1,1))}\,\sqrt{2\,n+1}\,\varepsilon^{\frac{1}{2}-s}\,\sqrt{2\,s-1}}{\displaystyle n^\frac{1}{2}\,\sqrt{2\,s-1}\,\|u\|_{L^2(A_\varepsilon)}}\\
&\le \sqrt{\frac{C}{s\,(1-s)}}\,\frac{\|u'\|_{L^2((-1,1))}^{s}}{\|u\|_{L^2(A_\varepsilon)}}\, \frac{n^{-s}}{\sqrt{2\,s-1}}+C\,\frac{\|u\|_{L^\infty((-1,1))}}{\|u\|_{L^2(A_\varepsilon)}}\,\sqrt{3}\,\varepsilon^{\frac{1}{2}-s}.
\end{split}
\]
It is now important to make a clever choice of the parameters $n$ and $\varepsilon$: we take them to be
\[
\varepsilon=\left(\frac{1}{10}\right)^\frac{1}{2\,s-1}\qquad \mbox{ and }\qquad n=\left(\left\lfloor \frac{1}{2\,s-1}\right\rfloor+1\right)^2.
\]
Observe that with these choices we have 
\[
\lim_{s\searrow \frac{1}{2}} \varepsilon=0\qquad \mbox{ and }\qquad \varepsilon^{\frac{1}{2}-s}=\sqrt{10},
\]
and
\[
\lim_{s\searrow \frac{1}{2}} \frac{n^{-s}}{\sqrt{2\,s-1}}\le \lim_{s\searrow \frac{1}{2}} (2\,s-1)^{2\,s-\frac{1}{2}}=0,
\]
where we also used \eqref{mantissa}. Moreover, by using the Dominated Convergence Theorem, we also have
\[
\lim_{s\searrow \frac{1}{2}}\|u\|_{L^2(A_\varepsilon)}=\|u\|_{L^2((-1,1))}=1.
\]
These facts finally enable us to conclude that
\[
\limsup_{s\searrow \frac{1}{2}}\sqrt{\frac{\lambda_1^s(\mathbb{R}\setminus\mathbb{Z})}{2\,s-1}}\le \sqrt{30}\,C\,\|u\|_{L^\infty((-1,1))}<+\infty.
\]
The proof is now over.	

\appendix

\section{A bi-Lipschitz homeomorphism}
\label{app:A}

In what follows, for every open set bounded $K\subseteq\mathbb{R}^N$ and every $x_0\in K$, we define 
\[
d_K(x_0)=\min_{x\in\partial K} |x-x_0|,\qquad D_K(x_0)=\max_{x\in \partial K} |x-x_0|.
\]
\begin{lemma}\label{lemma: cambio di variabili bi-lipschitz}
	Let $K\subseteq\mathbb{R}^N$ be an open bounded convex set and $x_0\in K$. There exists a bi-Lipschitz homeomorphism $\Phi_{K,x_0}:\mathbb{R}^N\to \mathbb{R}^N$ with the following properties: 
\begin{itemize}
\item $\Phi_{K,x_0}(x_0)=x_0$ and $\Phi_{K,x_0}(r\,(K-x_0)+x_0)=B_r(x_0)$, for every $r>0$;
\vskip.2cm
\item $\Phi_{K,x_0}$ is $L_K-$Lipschitz with 
\[
L_K=\frac{2}{d_K(x_0)};
\]
\item $\Phi^{-1}_{K,x_0}$ is $M_K-$Lipschitz with 
\[ 
 M_K=D_K(x_0)\,\left(2+\frac{D_K(x_0)}{d_K(x_0)}\right).
\]
\end{itemize}
Moreover, we have
\begin{equation}
\label{jacobian}
	\left(\dfrac{1}{M_K}\right)^N\le|\mathrm{det} \nabla\Phi_{K,x_0}(x)|\le (L_K)^N,\qquad \mbox{ for a.\,e. }x\in \mathbb{R}^N.
\end{equation}	
and
\begin{equation}
\label{jacobianinv}
	\left(\dfrac{1}{L_K}\right)^N\le|\mathrm{det} \nabla\Phi^{-1}_{K,x_0}(x)|\le (M_K)^N,\qquad \mbox{ for a.\,e. }x\in \mathbb{R}^N.
\end{equation}	
\end{lemma}
\begin{proof}
For notational simplicity, we omit to indicate the subscript $x_0$ everywhere.
We define at first the {\it Minkowski functional of $K$ centered at $x_0$}, i.e.
\[
j_{K}(x)=\inf\Big\{\lambda>0\, :\, x\in \lambda\,(K-x_0)+x_0\Big\}.
\]
We recall that this a Lipschitz function, which verifies the following homogeneity property
\begin{equation}
\label{homo}
j_{K}(t\,(x-x_0)+x_0)=t\,j_{K}(x),\qquad \mbox{ for every } x\in\mathbb{R}^N,\ t>0.
\end{equation}
We also observe that by construction it holds
\[
j_{K}(x)<r\quad \mbox{ if and only if }\quad x\in r\,(K-x_0)+x_0,
\]
and that 
\[
j_{K}(x)=r\quad \mbox{ if and only if }\quad x\in r\,(\partial K-x_0)+x_0.
\]
Moreover, $j_{K}$ is Lipschitz and it holds
\begin{equation}
\label{lipj}
|j_K(x)-j_K(y)|\le \frac{1}{d_K(x_0)}\,|x-y|,\qquad \mbox{ for every } x,y\in\mathbb{R}^N.
\end{equation}
Last, but not least, we have the following lower bound
\begin{equation}
\label{lowerbound}
j_K(x)=|x-x_0|\,j_K\left(\frac{x-x_0}{|x-x_0|}+x_0\right)\ge \frac{|x-x_0|}{D_K(x_0)}.
\end{equation}
Then we define $\Phi_{K}$ as follows
\[
\Phi_K(x_0)=x_0,\qquad \Phi_K(x)=\frac{x-x_0}{|x-x_0|}\,j_K(x)+x_0,\qquad \mbox{ if } x\in \mathbb{R}^N\setminus\{x_0\}.
\]
Thanks to the properties of the Minkowski functional, we have that $\Phi_K$ is injective. In order to verify that $\Phi_K$ is bijective, let us take $y\in \mathbb{R}^N\setminus\{x_0\}$. 
We then define
\begin{equation}
\label{inverse}
\overline{x}=\frac{|y-x_0|}{j_K(y)}\,(y-x_0)+x_0,
\end{equation}
we claim that $\Phi_K(\overline{x})=y$. Indeed, by construction we have
\[
\Phi_K(\overline{x})=\frac{\overline{x}-x_0}{|\overline{x}-x_0|}\,j_K(\overline{x})+x_0=
\frac{y-x_0}{|y-x_0|}\,j_K\left(\frac{|y-x_0|}{j_K(y)}\,(y-x_0)+x_0\right)+x_0.
\]
From property \eqref{homo} we get
\[
\Phi_K(\overline{x})=\frac{y-x_0}{|y-x_0|}\,j_K\left(\frac{|y-x_0|}{j_K(y)}\,(y-x_0)+x_0\right)+x_0=\frac{y-x_0}{|y-x_0|}\,j_K(y)\, \frac{|y-x_0|}{j_K(y)}+x_0=y,
\]
as desired. This shows that $\Phi_K$ is bijective and from \eqref{inverse} we get
\[
\Phi_K^{-1}(y)=\frac{|y-x_0|}{j_K(y)}\,(y-x_0)+x_0,\qquad \mbox{ for } y\in\mathbb{R}^N\setminus\{x_0\}.
\]
Thanks to the properties of the Minkowski functional, it is easily seen that 
\[
\Phi_K(r\,(K-x_0)+x_0)= B_r(x_0),\qquad \mbox{ for every } r>0.
\]
We now claim that both $\Phi_K$ and its inverse are Lipschitz continuous. We start with $\Phi_K$: we take $x,y\in \mathbb{R}^N\setminus\{x_0\}$ and, without loss of generality, we can suppose that $|y-x_0|\le |x-x_0|$. By the triangle inequality we get
\begin{equation}
\label{1lip}
\begin{split}
|\Phi_K(x)-\Phi_K(y)|&\le j_K(y)\,\left|\frac{x-x_0}{|x-x_0|}-\frac{y-x_0}{|y-x_0|}\right|+|j_K(x)-j_K(y)|\\
&\le j_K(y)\,\frac{|x-y|}{\sqrt{|x-x_0|\,|y-x_0|}}+|j_K(x)-j_K(y)|,
\end{split}
\end{equation}
where we used that
\[
\begin{split}
\left|\frac{x-x_0}{|x-x_0|}-\frac{y-x_0}{|y-x_0|}\right|^2&=2-2\,\frac{\langle x-x_0,y-x_0\rangle}{|x-x_0|\,|y-x_0|}\\
&\le \frac{|x-x_0|^2+|y-x_0|^2}{|x-x_0|\,|y-x_0|}-2\,\frac{\langle x-x_0,y-x_0\rangle}{|x-x_0|\,|y-x_0|}=\frac{|x-y|^2}{|x-x_0|\,|y-x_0|},
\end{split}
\]
thanks to Young's inequality. By using \eqref{lipj} and the fact that $|y-x_0|\le |x-x_0|$,
we get from \eqref{1lip}
\[
|\Phi_K(x)-\Phi_K(y)|\le \frac{1}{d_K(x_0)}\, \left[|y-x_0|\,\frac{|x-y|}{\sqrt{|x-x_0|\,|y-x_0|}}+|x-y|\right]\le \frac{2}{d_K(x_0)}\, |x-y|.
\]
This proves the claimed Lipschitz regularity of $\Phi_K$. 
\par
We now turn our attention to the inverse function $\Phi_K^{-1}$. We proceed in a similar way: we take $x,y\in \mathbb{R}^N\setminus\{x_0\}$ and we can suppose that $|y-x_0|\le |x-x_0|$. 
Then by the triangle inequality
\[
\begin{split}
|\Phi_K^{-1}(x)-\Phi_K^{-1}(y)|&\le \frac{1}{j_K(x)}\,\Big||x-x_0|\,(x-x_0)-|y-x_0|\,(y-x_0)\Big|\\
&+|y-x_0|^2\,\left|\frac{1}{j_K(x)}-\frac{1}{j_K(y)}\right|.
\end{split}
\]
By using \eqref{lowerbound} and observing that 
\[
\Big||x-x_0|\,(x-x_0)-|y-x_0|\,(y-x_0)\Big|\le (|x-x_0|+|y-x_0|)\,|x-y|=2\,|x-x_0|\,|x-y|,
\]
we get that
\[
\begin{split}
|\Phi_K^{-1}(x)-\Phi_K^{-1}(y)|&\le\frac{2\,|x-x_0|}{j_K(x)}\,|x-y|+\frac{|y-x_0|^2}{j_K(x)\,j_K(y)}\,|j_K(x)-j_K(y)|\\
&\le2\,D_K(x_0)\,|x-y|
+D_K(x_0)^2\,\frac{|y-x_0|^2}{|x-x_0|\,|y-x_0|}\,\frac{|x-y|}{d_K(x_0)}\\
&\le2\,D_K(x_0)\,|x-y|+D_K(x_0)^2\,\frac{|x-y|}{d_K(x_0)}\\
&=D_K(x_0)\,\left(2+\frac{D_K(x_0)}{d_K(x_0)}\right)\,|x-y|.
\end{split}
\]
This gives the desired Lipschitz estimate for $\Phi_K^{-1}$, as well.
\par
Finally, the two-sided estimates \eqref{jacobian} and \eqref{jacobianinv} are a standard consequence of the Lipschitz estimates on $\Phi_K$ and $\Phi_K^{-1}$, in conjunction with the Area Formula for Lipschitz functions and Rademacher's Theorem.
\end{proof}

\section{A special cut-off function}

\begin{lemma}
\label{lm:stimonaMFO}
Let $1/2<s<1$ and let 
\[
\zeta_s(x)=\Big(1-|x|^{2\,s-1}\Big)_+,\qquad \mbox{ for } x\in\mathbb{R}.
\]
Then we have
\begin{equation}
\label{stimonaMFO}
[\zeta_s]_{W^{s,2}(\mathbb{R})}\le C\,\frac{\sqrt{2\,s-1}}{\sqrt{1-s}},
\end{equation}
with a constant $C>0$ independent of $s\in(1/2,1)$.

\end{lemma}
\begin{proof}
We decompose the seminorm as follows
\begin{equation}
\label{3pezzi}
\begin{split}
[\zeta_s]^2_{W^{s,2}(\mathbb{R})}&=\iint_{(-1,1)\times (-1,1)} \frac{\Big||x|^{2\,s-1}-|y|^{2\,s-1}\Big|^2}{|x-y|^{1+2\,s}}\,dx\,dy\\
&+\frac{1}{s}\,\int_{-1}^1 \frac{\Big|1-|x|^{2\,s-1}\Big|^2}{(1-x)^{2\,s}}\,dx+\frac{1}{s}\,\int_{-1}^1 \frac{\Big|1-|x|^{2\,s-1}\Big|^2}{(1+x)^{2\,s}}\,dx=\mathcal{I}_1+\mathcal{I}_2+\mathcal{I}_3.
\end{split}
\end{equation}
In order to prove \eqref{stimonaMFO}, we will prove that 
\begin{equation}
\label{stimonaMFOpezzi}
\mathcal{I}_i\le C\,\frac{2\,s-1}{1-s},\qquad \mbox{ for }i=1,2,3.
\end{equation}
For the first term $\mathcal{I}_1$, we observe that by using the symmetry of the set and of the integrand, we have
\[
\mathcal{I}_1\le 4\,\iint_{(0,1)\times (0,1)} \frac{\Big||x|^{2\,s-1}-|y|^{2\,s-1}\Big|^2}{|x-y|^{1+2\,s}}\,dx\,dy.
\]
By using \cite[Remark 4.2, formula (4.3)]{BBZ} with the choice $\beta=2\,s-1$ there, we can estimate the last double integral as follows
\[
\mathcal{I}_1\le 4\,\iint_{(0,1)\times (0,1)} \frac{\Big||x|^{2\,s-1}-|y|^{2\,s-1}\Big|^2}{|x-y|^{1+2\,s}}\,dx\,dy\le \left(\int_0^1 \frac{|1-\tau^{2\,s-1}|^2}{|1-\tau|^{1+2\,s}}\,\left(1+\tau^{1-2\,s}\right)\,d\tau\right)\frac{4}{2\,s-1}.
\]
We then write
\[
\begin{split}
\int_0^1 \frac{|1-\tau^{2\,s-1}|^2}{|1-\tau|^{1+2\,s}}\,\left(1+\tau^{1-2\,s}\right)\,d\tau&=\int_0^\frac{1}{2} \frac{|1-\tau^{2\,s-1}|^2}{|1-\tau|^{1+2\,s}}\,\left(1+\tau^{1-2\,s}\right)\,d\tau\\
&+\int_\frac{1}{2}^1 \frac{|1-\tau^{2\,s-1}|^2}{|1-\tau|^{1+2\,s}}\,\left(1+\tau^{1-2\,s}\right)\,d\tau\\
&\le C\, \int_0^\frac{1}{2} |1-\tau^{2\,s-1}|^2\,\left(1+\tau^{1-2\,s}\right)\,d\tau\\
&+C\,\int_\frac{1}{2}^1 \frac{|1-\tau^{2\,s-1}|^2}{|1-\tau|^{1+2\,s}}\,d\tau=:\mathcal{I}_{1,1}+\mathcal{I}_{1,2}.
\end{split}
\]
The constant $C>0$ can be taken independent of $s$.
We start by estimating $\mathcal{I}_{1,2}$, which is simpler: we use the following pointwise inequality  
\[
a^\alpha-b^\alpha\le \alpha\,b^{\alpha-1}\,(a-b),\qquad \mbox{ for } 0<b\le a, \, 0<\alpha<1,
\]
which just follows from concavity of the map $\tau\mapsto \tau^{\alpha}$. This gives 
\[
\mathcal{I}_{1,2}\le C\,16^{1-s}\,(2\,s-1)^2\,\int_\frac{1}{2}^1 (1-\tau)^{1-2\,s}\,d\tau=\frac{C}{2\,(1-s)}\,4^{1-s}\,(2\,s-1)^2,
\]
as desired. We now come to $\mathcal{I}_{1,1}$, which is the most subtle. We set for simplicity
\[
f_\tau(s)=\tau^{2\,s-1},\qquad \mbox{ for } \tau>0,\, s>\frac{1}{2}.
\]
Then we have 
\[
\left|f_\tau(s)-f_\tau\left(\frac{1}{2}\right)\right|=\left|\int^s_\frac{1}{2} f'_\tau(t)\,dt\right|,
\]
that is 
\[
|1-\tau^{2\,s-1}|=2\,|\log \tau|\,\left|\int_\frac{1}{2}^s \tau^{2\,t-1}\,dt\right|\le 2\,(-\log \tau)\,\left(s-\frac{1}{2}\right)=(-\log \tau)\,(2\,s-1).
\]
Thus we get
\[
\mathcal{I}_{1,1}\le C\,(2\,s-1)^2\,\int_0^\frac{1}{2} (-\log \tau)^2\,(1+\tau^{2\,s-1})\,d\tau\le 2\,C\,(2\,s-1)^2\,\int_0^\frac{1}{2} (-\log \tau)^2\,d\tau.
\]
This gives the desired estimate, since the last integral is finite. By collecting the estimates for $\mathcal{I}_{1,1}$ and $\mathcal{I}_{1,2}$, we thus get \eqref{stimonaMFO} for $\mathcal{I}_1$.
\par
We now  consider $\mathcal{I}_2$ and $\mathcal{I}_3$. We only estimate the first one, since the estimate for the second one is similar. For $s>1/2$, we have
\[
\begin{split}
\frac{1}{s}\,\int_{-1}^1 \frac{\Big|1-|x|^{2\,s-1}\Big|^2}{(1-x)^{2\,s}}\,dx&\le 2\,\int_0^1 \frac{(1-x^{2\,s-1})^2}{(1-x)^{2\,s}}\,dx+2\,\int_{-1}^0\frac{(1-|x|^{2\,s-1})^2}{(1-x)^{2\,s}}\,dx\\
&\le 2\,\int_0^1 \frac{(1-x^{2\,s-1})^2}{(1-x)^{2\,s}}\,dx+2\,\int_{-1}^0 (1-|x|^{2\,s-1})^2\,dx\\
&= 2\,\int_\frac{1}{2}^1 \frac{(1-x^{2\,s-1})^2}{(1-x)^{2\,s}}\,dx+2\,\int_0^\frac{1}{2} \frac{(1-x^{2\,s-1})^2}{(1-x)^{2\,s}}\,dx \\
&+2\,\int_{0}^1 (1-x^{2\,s-1})^2\,dx\\
&\le 2\,\int_\frac{1}{2}^1 \frac{(1-x^{2\,s-1})^2}{(1-x)^{2\,s}}\,dx+2\cdot 4^s\,\int_0^\frac{1}{2} (1-x^{2\,s-1})^2\,dx\\
&+2\,\int_{0}^1 (1-x^{2\,s-1})^2\,dx\\
&\le 2\,\int_\frac{1}{2}^1 \frac{(1-x^{2\,s-1})^2}{(1-x)^{2\,s}}\,dx+2\,(4^s+1)\,\int_0^1 (1-x^{2\,s-1})\,dx.
\end{split}
\]
By computing the last integral, this gives in particular
\[
\frac{1}{s}\,\int_{-1}^1 \frac{\Big|1-|x|^{2\,s-1}\Big|^2}{(1-x)^{2\,s}}\,dx\le 2\,\int_\frac{1}{2}^1 \frac{(1-x^{2\,s-1})^2}{(1-x)^{2\,s}}\,dx+(4^s+1)\,\frac{2\,s-1}{s}.
\]
At this point, the integral in the right-hand side can be estimated as we did for $\mathcal{I}_{1,2}$ above. By proceeding as before, we get \eqref{stimonaMFOpezzi} for $\mathcal{I}_2$ (and thus for $\mathcal{I}_3$), as well. 
\end{proof}

\end{document}